\documentclass[10pt,reqno]{amsart}
\usepackage{geometry}
\usepackage{amsmath}
\usepackage[utf8]{inputenc}
\usepackage{verbatim}
\usepackage{dsfont}
\usepackage[toc,page]{appendix}
\usepackage{accents}
\usepackage{xcolor}
\usepackage{stmaryrd}
\usepackage{esint}

\newcommand{\setR}{{\mathbb R}}
\newcommand{\setN}{{\mathbb N}}
\newcommand{\ent}{\mathcal H}
\newcommand{\vN}{\mathtt{H}}
\newcommand{\fish}{\mathcal F}
\newcommand{\nrg}{\mathcal E}
\newcommand{\fnc}{{\mathcal U}}
\newcommand{\aux}{{\mathcal V}}
\newcommand{\flow}{\mathcal S}
\newcommand{\nrgm}{\overline{\mathcal E}}
\newcommand{\dn}{\mathrm{d}}
\newcommand{\dd}{\,\mathrm{d}}
\newcommand{\prop}{\mathcal{P}(\mathbb{R}^d)}
\newcommand{\propo}{\mathcal{P}_2(\mathbb{R}^d)}
\newcommand{\wass}{{\mathbf W}_2}
\newcommand{\loc}{\text{loc}}
\newcommand{\vpot}{\Phi}
\newcommand{\mom}{\mathfrak m_2}
\newcommand{\eins}{\mathbb{I}}
\newcommand{\dff}{\mathrm{D}}
\newcommand{\hess}{\mathrm{Hess}}
\newcommand{\dv}{\mathrm{div}}
\newcommand{\theL}{{\mathbf L}}
\newcommand{\velo}{\mathbf v}
\newcommand{\heat}{{\mathcal K}}

\newcommand{\norm}[1]{\left\lVert#1\right\rVert}
\newcommand{\norms}[1]{\left\lvert#1\right\rvert}

\newcommand{\rn}{\mathbb{R}^\dimens}

\newcommand{\intrn}{\int_{\mathbb{R}^\dimens}}

\newcommand{\one}{\mathds{1}}
\newcommand{\id}{\textnormal{id}}

\newcommand{\ntr}{\operatorname{tr}}
\newcommand{\leb}{\mathcal{L}^\dimens}
\newcommand{\propp}{\mathcal{P}}
\newcommand{\Dom}[1]{\mathrm{Dom}(#1)}

\newcommand{\wasser}[2]{\textnormal{\bfseries W}_2^2(#1,#2)}
\newcommand{\wassernoq}[2]{\textnormal{\bfseries W}_2(#1,#2)}

\newcommand{\s}{\sigma}
 
\newcommand{\N}[2]{\mathcal{N}[#1,#2]}

\newcommand{\moment}{\textbf{m}}
\newcommand{\eps}{\varepsilon}
\newcommand{\dimens}{d}

\newcommand{\ebnd}{{\mathrm E_\lambda}}
\newcommand{\lip}{\mathrm L}

\renewcommand{\phi}{\varphi}
\renewcommand{\epsilon}{\varepsilon}

\newtheorem{theorem}{Theorem}[section]
\newtheorem{definition}[theorem]{Definition}
\newtheorem{corollary}[theorem]{Corollary}
\newtheorem{lemma}[theorem]{Lemma}
\newtheorem{proposition}[theorem]{Proposition}
\newtheorem{remark}[theorem]{Remark}
\newtheorem{conjecture}{Conjecture}
\newtheorem{example}{Example}

\newcommand\blfootnote[1]{%
	\begingroup
	\renewcommand\thefootnote{}\footnote{#1}%
	\addtocounter{footnote}{-1}%
	\endgroup
}

\title[Sixth Order QDD]{Gradient Flow Structure of a Multidimensional Nonlinear Sixth Order Quantum-Diffusion Equation}
\author{Daniel Matthes}
\email{matthes@ma.tum.de}
\address{Zentrum Mathematik, TU München, Boltzmannstrasse 3, D-85748 Garching, Germany}
\author{Eva-Maria Rott}
\address {Zentrum Mathematik, TU München, Boltzmannstrasse 3, D-85748 Garching, Germany}
\email{eva-maria.rott@tum.de}
\date{\today}

\begin{document}

\begin{abstract}
  A nonlinear parabolic equation of sixth order is analyzed.
  The equation arises as a reduction of a model from quantum statistical mechanics,
  and also as the gradient flow of a second-order information functional with respect to the $L^2$-Wasserstein metric. 
  First, we prove global existence of weak solutions for initial conditions of finite entropy
  by means of the time-discrete minimizing movement scheme.
  Second, we calculate the linearization of the dynamics around the unique stationary solution,
  for which we can explicitly compute the entire spectrum.
  A key element in our approach is a particular relation between the entropy, the Fisher information
  and the second order functional that generates the gradient flow under consideration.
\end{abstract}
\blfootnote{
  This research was supported by the DFG Collaborative Research Center TRR 109, ``Discretization in Geometry and Dynamic''.
}
\maketitle

\tableofcontents

\section{Introduction}

\subsection{The equation}
The following nonlinear parabolic evolution equation of sixth order is considered:
\begin{align}\label{eq:QDD6intro}
  \partial_t u = \dv \big( u \,\nabla\big(\vpot(u) + \lambda^3|x|^2\big)\big),
  \quad \vpot(u) =
  \frac12 \norm{\nabla^2 \log u}^2+ \frac{1}{u} \nabla^2 : (u \nabla^2 \log u),
\end{align}
where $\lambda\ge0$ is a given parameter.
There are at least two different contexts in which \eqref{eq:QDD6intro} plays a role.

The first is the semi-classical approximation of the nonlocal quantum drift-diffusion model by Degond et al \cite{DegMeRi05}.
In the formal asymptotic expansion of that equation in terms of the Planck constant $\hbar$,
the right-hand side of \eqref{eq:QDD6intro} with $\lambda=0$ appears as the term of order $\hbar^4$,
after the linear diffusion operator $\Delta u$ at order $\hbar^0$
and the operator $-\nabla^2:(u\nabla^2\log u)$, which is related to the Bohm potential, at order $\hbar^2$.
More details on the derivation of the model and its formal expansion are given below in Section \ref{sct:derivation}.

The other context, more relevant to the paper at hand,
is that of gradient flows in the $L^2$-Wasserstein distance.
Consider the following three functionals,
defined on --- at the moment for simplicity: strictly positive --- probability densities $u:\setR^\dimens\to\setR$
by
\begin{equation}
  \label{eq:HFE}
  \begin{split}
    \ent_\lambda(u) &= \intrn u\left(\log u+\frac\lambda2|x|^2\right)\dd x, \\
    \fish_\lambda(u) &= \frac12\intrn u\left(|\nabla\log u|^2+\lambda^2|x|^2\right)\dd x, \\
    \nrg_\lambda(u) &= \frac12\intrn u\left(\|\nabla^2\log u\|^2+2\lambda^3|x|^2\right)\dd x,
  \end{split}
\end{equation}
which we shall refer to as --- \emph{$\lambda$-perturbed} if $\lambda>0$ --- \emph{entropy}, \emph{Fisher information}, and \emph{energy}, respectively.
The celebrated result of \cite{JKO} is that
the gradient flow of $\ent_\lambda$ in the $L^2$-Wasserstein metric is the linear Fokker-Planck equation,
\begin{align}
  \label{eq:FP}
  \partial_tu = \Delta u + \lambda\dv(xu).
\end{align}
In \cite{GST} --- see also\cite[Example 11.1.10]{AGS} --- it has been shown that the gradient flow of $\fish_\lambda$ is
the so-called quantum drift-diffusion or DLSS equation,
\begin{align}
  \label{eq:DLSS}
  \partial_tu = -\nabla^2:(u\nabla^2\log u) + \lambda\dv(xu).
\end{align}
The starting point for our analysis is that \eqref{eq:QDD6intro} is the gradient flow of $\nrg_\lambda$,
at least formally, that is, the potential $\vpot$ in \eqref{eq:QDD6intro} is the variational derivative of $\nrg_\lambda$.
The reason for considering $\nrg_\lambda$ as potential for a gradient flow
is more profound than its formal similarity with $\ent_\lambda$ and $\fish_\lambda$.
There is an intimate relation between $\ent_\lambda$, $\fish_\lambda$ and $\nrg_\lambda$,
that has already been the basis for deriving sharp self-similar asymptotics in \cite{MMS},
and that we shall elaborate on in Section \ref{sct:int-hessian} below.
In a nutshell, the dissipation of $\ent_\lambda$ along the heat flow is $\fish_\lambda$,
the dissipation of $\fish_\lambda$ along the heat flow is $\nrg_\lambda$, up to a multiple of $\fish_\lambda$ itself,
and this equals the dissipation of $\ent_\lambda$ along the flow of \eqref{eq:DLSS}.
In this spirit, one may consider \eqref{eq:DLSS} and \eqref{eq:QDD6intro}, respectively,
as fourth and sixth order analogues of the second order linear Fokker-Planck equation \eqref{eq:FP}.
\smallskip

Our analytical results are two-fold.
First, we give a proof of existence of weak solutions to the initial value problem for \eqref{eq:QDD6intro}
on the whole space $\setR^d$ for initial data with finite entropy and finite second moment.
This result is proven with full rigor.
Second, we study the long-time asymptotics of solutions using a linearization around the steady state.
This part is formal in the sense that we calculate the linearization for sufficiently smooth perturbations of the steady state
and discuss the spectral properties of an appropriate closure of the linear operator.


\subsection{Existence of solutions}
Global existence of non-negative weak solutions to \eqref{eq:QDD6intro} with $\lambda=0$
on the $\dimens$-dimensional torus, i.e., with periodic boundary conditions,
has been proven in \cite{JuMi} for $\dimens=1$,
and in \cite{BJM} for $\dimens=2$ and $\dimens=3$.
The main technical ingredient of these proofs is a particular regularization of the evolution equation,
namely by $\epsilon\Delta^3\log u$.
This regularization produces approximations of the true solution
that are smooth and have an $\epsilon$-dependent positive lower bound.
Smoothness of both $u$ and $\log u$ then allows to perform all the necessary a priori estimates
on the approximation, and these pass to the limit $\epsilon\downarrow0$.
An extension of this method from the torus to the whole space --- if possible at all --- is at least not straight-forward,
since the intermediate estimates involve simultaneous Sobolev estimates on $u$ and $\log u$,
that would contradict each other on unbounded domains.

Here, we shall prove existence of solutions in $\dimens\le3$ on the whole space $\setR^d$.
Our technical device is very different from the one recalled above:
we invoke the machinery of metric gradient flows.
Our approximations are of lower regularity, and we do not have any information about positivity.
To justify the a priori estimates, we need the method of flow interchange with the heat equation \cite{GST,MMS}.
In contrast to the constructions in \cite{JuMi,BJM}, we do not modify the equation,
but perform a variational discretization in time using the minimizing movement scheme.
More specifically: given initial probability density $u_0$ of finite second moment and entropy,
and a time step $\tau>0$,
define inductively a sequence $(u_\tau^n)_{n=0}^\infty$ by $u_\tau^0=u_0$,
and $u_\tau^n$ being a minimizer of
\begin{align}
  \label{eq:MMintro}
  u\mapsto \frac1{2\tau}\wass(u,u_\tau^{n-1})^2 + \nrg_\lambda(u).
\end{align}
We recall the $L^2$-Wasserstein-distance $\wass$ below in Section \ref{sct:background},
and we prove that the inductive procedure is well-defined.
Our result about existence is that the sequences $(u_\tau^n)$ approximate a weak solution to \eqref{eq:QDD6intro}.
\begin{theorem}
  \label{thm:existence}
  Assume $\dimens \leq 3$.
  Let an initial datum $u_0$ be given that is a probability density with finite entropy, $\ent_0(u_0)<\infty$,
  and finite second moment.
  For each $\tau>0$, define a sequence $(u_\tau^n)_{n=0}^\infty$ inductively as described above.
  Then the piecewise constant ``interpolations'' $\bar u_\tau:\setR_{\ge0}\to L^1(\setR^d)$
  with
  \begin{align}
    \label{eq:interpol}
    \bar u_\tau(t)=u_\tau^n \quad \text{for all $(n-1)\tau<t\le n\tau$}    
  \end{align}
  converge along a suitable null sequence $(\tau_k)$ to a limit $u:\setR_{\ge0}\times\setR^d\to\setR$ in $L^2_\loc(\mathbb{R}_{>0};W^{2,2}(\setR^d))$. 
  And that limit $u$ is a weak solution to \eqref{eq:QDD6intro} with $u(0)=u_0$ in the following sense:
  $t\mapsto u(t,\cdot)\leb$ is a locally H\"older continuous curve 
  from $\setR_{\ge0}$ in $(\propo, \wass)$, 
  the roots $s=\sqrt{u}$ and $z=\sqrt[4]{u}$ are of regularity
  \begin{align}
    \label{eq:rootregular}
    s\in L^2_\loc(\setR_{>0};W^{2,2}(\setR^d)), \quad z\in L^4_\loc(\setR_{>0};W^{1,4}(\setR^d)),
  \end{align}
  respectively, and
  for every test function $\varphi\in C^\infty_c(\setR_{>0}\times\setR^d)$,
  \begin{equation}
    \label{eq:weak}
    \int_{0}^{\infty} \intrn \big(\partial_t \varphi-2\lambda^3x\cdot\nabla\varphi\big) u \dd x\dd t
    = \int_{0}^{\infty} \N{u}{\varphi} \dd t
  \end{equation} 
  where the nonlinearity is given by
  \begin{align}
    \label{eq:N}
    \N{u}{\varphi} = -4\intrn\big\{
    2\nabla^2\varphi:(\ell^2) + \frac12(s\nabla^2\Delta\varphi+2\nabla s\cdot\nabla^3\varphi):\ell 
    \big\}\dd x,
  \end{align}
  with $\ell=\nabla^2s-4\nabla z\otimes\nabla z$,
  which coincides with $\frac12\sqrt u\,\nabla^2\log u$ where $u>0$.
\end{theorem}
\begin{remark}
  In view of \eqref{eq:rootregular}, one has $\ell\in L^2_\loc(\setR_{>0}\times\setR^d)$,
  and so $\varphi$ is ``tested'' against a local $L^1$-function in \eqref{eq:N}.
  It is far from obvious that \eqref{eq:weak} with the nonlinearity \eqref{eq:N} is indeed
  a weak formulation of \eqref{eq:QDD6intro}.
  Equality of $- \N{u}{\varphi}$ with a more ``natural'' weak formulation of the nonlinearity in  \eqref{eq:QDD6intro}'s right hand side,
  like
  \begin{align*}
    \intrn\varphi\,\dv\big(u\nabla\Phi(u)\big)\dd x
    = \intrn \dv(u\nabla\varphi)\Phi(u)\dd x,
  \end{align*}
  for smooth and positive solutions $u$ can be verified by a direct but tedious
  computation involving various integration by parts.
  A more conceptual way to recognize $- \N{u}{\varphi}$ as a weak formulation
  is explained at the beginning of Section \ref{sct:deq}.
\end{remark}
Theorem \ref{thm:existence} is proven by time-discrete approximation of the solution $u$
via the celebrated minimizing movement scheme.
The main compactness estimate for passing to the time-continuous limit is provided
by the dissipation of the (unperturbed) entropy $\ent_0$, which formally amounts to
\begin{equation}
  \label{eq:apriori}
  -\frac{\dn}{\dd t}\ent_0(u) \ge
  \kappa \intrn \left(\interleave \nabla^3 \sqrt{u}\interleave^2 + \norms{\nabla \sqrt[6]{u}}^6\right) \dd x,
\end{equation}
with some positive $\kappa>0$.
The formal calculations leading to \eqref{eq:apriori} via integration by parts
are identical to the ones used in \cite{BJM}.
The justification of these estimates in the whole-space case is technically more involved.


\subsection{Long-time asymptotics}
\label{sct:int-hessian}
The interpretation of \eqref{eq:QDD6intro} as a Wasserstein gradient flow is essential for our second main result,
which concerns the long time asymptotics of $u$.
First assume $\lambda>0$, in which case there is an unique equilibrium $U_\lambda$ for \eqref{eq:QDD6intro} (with mass equal to 1),
given by
\begin{align}
  \label{eq:Ulambda}
  U_\lambda(x) = \left(\frac\lambda{2\pi}\right)^{d/2}\exp\left(-\frac\lambda2|x|^2\right).
\end{align}
Our approach to understanding the dynamics near $U_\lambda$
is to formally calculate a linearization of \eqref{eq:QDD6intro} at $U_\lambda$ and to determine its spectrum.
We wrote \emph{a} linearization since there are several competing concepts for linearization in nonlinear diffusion equations,
providing different pieces of information about the long-time asymptotics. 
Following the ideas of \cite{DMc}, we study here the ``linearization in Wasserstein'',
which is given by the so-called \emph{displacement Hessian} of $\nrg_\lambda$ at $U_\lambda$.
Very informally, the displacement Hessian of a functional $\fnc$ at a critical point $U_*$
is the representation of $\fnc$'s second variational derivative
with respect to the scalar product in $H^{-1}(\setR^d;U_*\dd\leb)$.
A definition and a more intuitive interpretation in terms of Lagrangian maps is given in Section \ref{sct:disphess}.

Displacement Hessians are rarely used for the analysis of long-time asymptotics
since the extraction of rigorous analytical information requires a lot
of a priori knowledge about regularity from the solution.
Particularly when it comes to proving higher order asymptotics,
alternative linearizations are often easier to handle;
we refer to the discussions in \cite{SeisTF,Koch}.
The most famous application of displacement Hessians
concerns the result on self-similar asymptotics for the porous medium equation \cite{Ot01};
further applications, also to higher order diffusion equations, can be found e.g. in \cite{McS,Slep,Seis}.
In the situation at hand, we use the Wasserstein linearization because of its compatibility
with the special structure of \eqref{eq:QDD6intro} that we outline below.

Since the derivation of the result is probably more interesting than the result itself,
we briefly indicate the main ingredient,
which is the aforementioned intimate relation between the three functionals in \eqref{eq:HFE}:
recall that the linear Fokker-Planck equation \eqref{eq:FP} is the Wasserstein gradient flow of $\ent_\lambda$.
Consider a solution $(w_r)_{r\ge0}$ to that flow, i.e.,
\begin{align}
  \label{eq:LFP}
  \partial_rw_r = \Delta w_r + \lambda\dv(xw_r).
\end{align}
It is a well-known fact that $\fish_\lambda$
is the dissipation of $\ent_\lambda$ along \eqref{eq:LFP}, i.e.,
\begin{align}
  \label{eq:lessmagic}
  \fish_\lambda(w_r)-\fish_\lambda(U_\lambda) = -\frac12\frac{\dn}{\dd r}\ent_\lambda(w_r).
\end{align}
The connection between $\ent_\lambda$ and $\nrg_\lambda$ --- and, implicitly, also $\fish_\lambda$ ---
is given by
\begin{align}
  \label{eq:justmagic}
  \nrg_\lambda(w_r) -\nrg_\lambda(U_\lambda)
  = \frac14\frac{\dn^2}{\dd r^2}\ent_\lambda(w_r) - \frac\lambda2\frac{\dn}{\dd r}\ent_\lambda(w_r).
\end{align}
The relation \eqref{eq:justmagic} is derived in Section \ref{sct:HFE}.
It will play the same role in our studies on \eqref{eq:QDD6intro}
as \eqref{eq:lessmagic} has played for the analysis of the DLSS equation \eqref{eq:DLSS} in \cite{McS}.
That is, we use \eqref{eq:justmagic} to express the displacement Hessian of $\nrg_\lambda$
in terms of the displacement Hessian of $\ent_\lambda$. 

For the linear Fokker-Planck equation \eqref{eq:LFP},
which is the gradient flow of the relative entropy $\ent_\lambda$ and has $U_\lambda$ as a critical point as well,
the displacement Hessian has been calculated in \cite[Proposition 2]{DMc}:
it is the extension of the formal operator
\begin{align}
  \label{eq:theL}
  \theL_\lambda\varphi = -\frac1{U_\lambda}\dv\big(U_\lambda\,\nabla\varphi\big) =- \Delta \varphi + \lambda x\cdot\nabla\varphi
\end{align}
to $W^{1,2}(\setR^d;U_\lambda\dd\leb)$.
The corresponding spectrum of $\theL_\lambda$ is well-known:
it is purely discrete with eigenvalues that are precisely the positive integer multiples of $\lambda$.
The corresponding eigenfunctions are Hermite polynomials.
The derivation of $\theL_\lambda$ might appear a bit weird:
first, one re-writes the linear Fokker-Planck equation \eqref{eq:LFP} on the $L^2$-Wasserstein space,
obtaining a highly non-linear gradient flow,
and then calculates its linearization,
which is --- up to dualization --- again the original equation \eqref{eq:LFP}.
The key point is that the Wasserstein linearization is compatible with the relation \eqref{eq:justmagic},
i.e., the linearization of the highly non-linear sixth order equation \eqref{eq:QDD6intro}
is easily expressible in terms of $\theL_\lambda$, see Theorem \ref{thm:hessian} below.
A similar idea has been used in \cite{McS},
where the authors showed by exploiting the relation \eqref{eq:lessmagic} below
that the displacement Hessian for $\fish_\lambda$ is formally given by $\theL_\lambda^2$,
and discussed implications on the long-time asymptotics of the nonlinear fourth order DLSS \eqref{eq:DLSS} equations.

Our second main result is:
\begin{theorem}
  \label{thm:hessian}
  Given a test function $\psi\in C^\infty_c(\setR^d)$,
  let $u_\s$ be the solution of the transport equation
  \begin{align*}
    \partial_\s u_\s = -\dv(u_\s\nabla\psi) \quad \text{subject to the initial condition $u_0=U_\lambda$}.
  \end{align*}
  Then:
  \begin{align*}
  \frac{\dn^2}{\dd\s^2}\bigg|_{\s=0}\nrg_\lambda(u_\s) = \intrn \nabla\psi\cdot\nabla\big(\theL_\lambda^3\psi\big)U_\lambda\dd\leb + \lambda \intrn \nabla\psi\cdot\nabla\big(\theL_\lambda^2\psi\big)U_\lambda\dd\leb.
  \end{align*}
  Consequently, if $\nrg_\lambda$ possesses a displacement Hessian at $U_\lambda$,
  then it is an extension of the linear operator $\theL_\lambda^3+\lambda\theL_\lambda^2$.
\end{theorem}
Theorem \ref{thm:hessian} could be proven by direct calculations,
evaluating the second variational derivative of $\nrg_\lambda$ along the transport equation,
and then integrating by parts until the desired form is attained.
This would be tedious but in principle straight-forward
once the desired terminal form $\theL_\lambda^3+\lambda\theL_\lambda^2$ is known.
Here we present a more conceptual approach, based on the relation \eqref{eq:justmagic},
which leads us to the form of the Hessian in the first place.

As said before, the implications of Theorem \ref{thm:hessian} on the long-time asymptotics of \eqref{eq:QDD6intro} are far from obvious.
Naturally, the hope is that the dynamics of all sufficiently smooth solutions $u$ close to equilibrium
is approximately the same as that of the linearized equation,
and in particular, that the eigenvalues of $\theL_\lambda^3 + \lambda\theL_\lambda^2$ determine
the exponential decay of $u$'s ``nonlinear modes'' in the long-time limit.
It is not hard to see that the formal differential operator $\theL_\lambda^3 + \lambda\theL_\lambda^2$ possesses
a unique self-adjoint closure in $W^{1,2}(\setR^d;U_\lambda\dd\leb)$,
and that the spectrum of the latter is purely discrete with eigenvalues $\lambda^3(k^3+k^2)$ for $k=0,1,2,\ldots$

The strong results from \cite{MMS},
where the fully nonlinear exponential stability of the Gaussian steady state for the DLSS equation \eqref{eq:DLSS}
has been proven on grounds of \eqref{eq:lessmagic},
give rise to a conjecture, namely that the spectral gap $2\lambda^3$ in the linearization
actually determines the global rate of convergence to equilibrium.
\begin{conjecture}
  \label{cjt:two}
  The weak solutions to \eqref{eq:QDD6intro} constructed in the course of the proof of Theorem \ref{thm:existence} above
  converge to $U_\lambda$ in $L^1(\setR^d)$ at an exponential rate of $\lambda^3$,
  i.e., there exists a constant $C(u^0)$ that is expressible in terms of the initial datum $u^0$ alone
  such that
  \begin{align}
    \label{eq:5}
    \|u(t,\cdot)-U_\lambda\|_{L^1(\setR^d)} \le C(u^0)e^{-\lambda^3t} \quad\text{for all $t\ge0$}.
  \end{align}
\end{conjecture}
A direct consequence of Conjecture \ref{cjt:two} would be $u$'s \emph{asymptotic self-similarity} in case $\lambda=0$.
More precisely, in Section \ref{sct:selfsim} we show that if Conjecture \ref{cjt:two} were true,
then any solution $u$ to \eqref{eq:QDD6intro} with $\lambda=0$ approaches in $L^1(\setR^d)$ with algebraic rate $t^{-1/6}$
the self-similar solution
\begin{align}
  \label{eq:selfsim}
  u_*(t;x) = [1+12t]^{-d/6}U_1\big([1+12t]^{-1/6}x\big)
\end{align}
at least if $u$ is already sufficiently close to self-similarity initially.

\bigskip

\paragraph{\bf Outline}
After explaining the origin and relevance of equation (\ref{eq:QDD6intro}) in Section \ref{sct:derivation},
we introduce commonly known background
regarding the $L^2$-Wasserstein space and metric gradient flows in Section \ref{sct:background}.
Section \ref{sct:existence} is then devoted to the existence proof,
while Section \ref{sct:longtime} deals with formal results for the intermediate respectively long time behaviour
of those obtained solutions.

\paragraph{\bf Notation}
 The euclidean norm is denoted by $\norms{\cdot}$, while $\norm{\cdot}$ is given by $\norm{A}^2 = \sum_{i,j=1}^{\dimens} a_{ij}^2$ for $A = (a_{ij}) \in \mathbb{R}^{\dimens \times \dimens}$ and $\interleave\mathbb{B}\interleave^2 = \sum_{i,j,k=1}^{\dimens} b_{ijk}^2$ for $\mathbb{B} = (b)_{ijk} \in \mathbb{R}^{\dimens \times \dimens \times \dimens}$. The domain $\Dom{{\mathcal{G}}}$ of a functional $\mathcal{G}$ defined on a set $X$ consists of all $u \in X$ such that $\mathcal{G}(u) < \infty$.


\section{Derivation and Preliminaries}
\label{sct:prelims}

\subsection{Derivation of the equation}
\label{sct:derivation}
  We sketch the derivation of \eqref{eq:QDD6intro} from a quantum mechanical model.
  In \cite{DegMeRi05} --- building on \cite{DegRing} --- the following
  non-linear and non-local quantum analogue of the classical Fokker-Planck equation \eqref{eq:LFP}
  has been derived:
  \begin{align}
    \label{eq:QDD}
    \partial_tu = \dv\left(u\,\nabla \left(A[u]+\frac\lambda2|x|^2\right)\right).
  \end{align}
  Here $u$ is the macroscopic density of quantum particles
  whose dynamics aims at minimizing the ensemble's relative von Neumann entropy $\vN_\lambda$.
  The precise definition of $A[u]$, sometimes referred to as $u$'s quantum logarithm, is intricate;
  for the sake of completeness, we mention one possible definition of $A[u]$ as
  the uniquely determined potential $A:\setR^d\to\setR$ such that
  \begin{align*}
    \operatorname{Tr}\left[\varphi\,\exp\left(-\frac{\hbar^2}2\Delta+A\right)\right]
    = \int_{\setR^d}\varphi(x)u(x)\dd x
  \end{align*}
  for all test functions $\varphi\in C^\infty_c(\setR^d)$.
  Here $\operatorname{Tr}$ denotes the trace over $L^2(\setR^d)$,
  and $\exp$ is the exponential of a self-adjoint operator.
  
  Under the specific hypotheses made in \cite{DegMeRi05},
  the von Neumann entropy can be expressed in terms of the macroscipic density $u$ alone,
  and is given by
  \[ \vN_\lambda(u)=\int_{\setR^d}u\left(A[u]+\frac\lambda2|x|^2\right)\dd x.\]
  We remark that $\vN_\lambda(u)$'s variational derivative is $A[u]+\frac\lambda2|x|^2$,
  and thus equation \eqref{eq:QDD} has the formal structure of a gradient flow in $\wass$.
  In the semi-classical limit $\hbar\to0$,
  the von Neumann entropy $\vN_\lambda(u)$ formally approaches the Boltzmann entropy $\ent_\lambda(u)$,
  and the variational derivative $A[u]$ formally approaches the pointwise logarithm $\log u$.
  Consequently, \eqref{eq:QDD} turns into the classical Fokker-Planck equation.

The existence analysis for the full equation \eqref{eq:QDD} goes far beyond classical parabolic theory,
and has been carried out just recently, and only in one space dimension \cite{Pi19}.
Already the rigorous definition of $A[u]$ as solution to an inverse problem has been challenging \cite{MePi10}.
A way to approximate \eqref{eq:QDD} by \emph{local} equations is via the expansion of $A[u]$ in powers of the small parameter $\hbar$.
In \cite[Appendix]{BuJuMa11} the following asymptotic expansion up to $\mathcal{O}(\hbar^6)$
has been computed (for $\lambda=0$):
\begin{equation*}
  A = \underbrace{\log u}_{=: A_0[u]}
  + \frac{\hbar^2}{12} \underbrace{\left( -2 \frac{\Delta \sqrt{u}}{\sqrt{u}} \right)}_{ =: A_1[u]}
  + \frac{\hbar^4}{360} \underbrace{\left(\frac{1}{2} \norm{\nabla^2 \log u}^2+ \frac{1}{u} \nabla^2 : (u \nabla^2 \log u) \right)}_{=: A_2[u]}
  + \mathcal{O}(\hbar^6).
\end{equation*}
As mentioned above, a reduction of $A[u]$ in \eqref{eq:QDD} to the leading order term $A_0[u]$
yields the linear Fokker-Planck (or rather: heat --- since $\lambda=0$) equation \eqref{eq:LFP},
\begin{align*}
  \partial_t u = \dv(u\,\nabla A_0[u]) = \Delta u.
\end{align*}
Replacing $A[u]$ by $A_1[u]$ leads to the Derida-Lebowitz-Speer-Spohn (DLSS) equation
\begin{align*}
  \partial_t u = \dv(u\nabla A_1[u]) =  -\nabla^2:(u\nabla^2\log u).
\end{align*}
Finally, since $A_2[u]$ is identical 
 to the functional $\Phi(u)$,
the equation $\partial_t u = \dv(u\nabla A_2[u])$ coincides
with the sixth-order equation \eqref{eq:QDD6intro} under consideration here.


\subsection{Background for Wasserstein gradient flows}
\label{sct:background}
In this section we briefly review some basics about the theory of optimal transport and $L^2$-Wasserstein gradient flows,
but only as far as it is needed later.
For a more profound introduction to these topics, we refer to the text books \cite{AGS,Santa,Vi03}.
By $\propo$ we denote all probability measures with finite second moment,
\begin{equation*}
  \moment_2(\mu) = \int_{\rn}|x|^2 \dd\mu(x) < \infty.
\end{equation*}
We shall frequently identify absolutely continuous (with respect to the Lebesgue measure $\leb$)
probability measures $\mu$ on $\setR^d$ with their densities $u\in L^1(\setR^d)$.
A sequence $(\mu_n)_{n \in \mathbb{N}} \subset \prop$ of probability measures converges \emph{narrowly} to $\rho \in \prop$
if
\begin{align*}
  \lim\limits_{n \rightarrow \infty} \int_{\rn}f(x) \dd\mu_n(x) = \int_{\rn}f(x) \dd\mu(x)
\end{align*}
holds for all bounded, continuous functions $f: \rn \rightarrow \mathbb{R}$.
The $L^2$-Wasserstein distance between two measures $\mu, \nu \in \propo$ is defined via 
\begin{equation}
  \label{eq:Wassersteindist1}
  \wasser{\mu}{\nu} = \min_{\pi \in \Pi(\mu, \nu)} \int_{\rn\times\rn} |x-y|^2 \dd\pi(x,y),
\end{equation}
where $\Pi(\mu, \nu)$ denotes the set of all transport plans between $\mu$ and $\nu$,
that is all $\pi\in \mathcal{P}(\rn \times \rn)$ with respective marginals $\mu$ and $\nu$.
The Wasserstein distance metrizes narrow convergence on $\propo$ and is itself lower semi-continuous with respect to that convergence.

We are not going to define metric gradient flows on $(\propo,\wassernoq\cdot\cdot)$ in general.
Here we just need a particularly nice subclass.
\begin{definition}
  Let $\aux:\propo\to\setR\cup\{+\infty\}$ be a proper, lower semi-continuous functional.
  Further, let a semi-group $(\flow_s)_{s\ge0}$ of continuous maps $\flow_s$ on $\propo$ be given.
  We call the semi-group an \emph{$\alpha$-flow for $\aux$}
  if the following \emph{evolutionary variational inequality} holds at any $s\ge0$ and with any $\mu,\nu\in\Dom\aux$:
  \begin{align}
    \label{eq:evi}
    \frac{1}{2} \frac{\dn^+}{\dd s} \wasser{\flow_s(\mu)}{\nu}
    + \frac{\alpha}{2}  \wasser{\flow_s(\mu)}{\nu}
    \le \aux(\nu) - \aux(\flow_s(\mu)).
    \end{align}
\end{definition}
\begin{example}
  \label{xmp:flow}
  The following three examples of $\alpha$-flows will be important in the following.
  They are both special cases of \cite[Example 11.2.7]{AGS}.
  \begin{enumerate}
  \item \emph{The linear heat flow}, given as distributional solution to $\partial_s\mu=\Delta\mu$,
    is a $0$-flow for the entropy $\aux=\ent_0$.
  \item \emph{The linear mass transport}, given as distributional solution to $\partial_s\mu=\dv(\mu\nabla V)$
    for a potential $V\in C^2(\setR^d)$ with globally bounded second derivatives,
    is an $\alpha$-flow for the potential energy $\aux(\mu) = \intrn V(x)\dd\mu(x)$,
    for every $\alpha$ such that $\nabla^2V\ge\alpha\eins$.
  \item \emph{The rescaling}, given as distributional solution to $\partial_s\mu=\dv(x\,\mu)$,
    is an $1$-flow for the potential energy $\aux=\frac12\mom$.
  \end{enumerate}
\end{example}
  Solutions to \eqref{eq:QDD6intro} will be constructed via discrete-in-time approximation
  by means of the minimizing movement scheme, i.e., a variational Euler method,
  see Section \ref{sct:mm} below.
  Inductively, the approximation $\mu_\tau^n$ at time $t=n\tau$ is obtained from $\mu_\tau^{n-1}$, the one at $t=(n-1)\tau$,
  as minimizer in
  \begin{align}
    \label{eq:mm-preview}
    \mu\mapsto \frac1{2\tau}\wass\big(\mu,\mu_\tau^{n-1}\big)^2+ \nrgm_\lambda(\mu).
  \end{align}
  For passage to the continuous limit, a priori estimates independent of the time step $\tau$ are essential.
  The key is to give a rigorous meaning to a priori estimates
  related to dissipations $-\frac{\dn}{\dd t}\aux(\mu_t)$ of auxiliary functionals $\aux$
  already on the time-discrete level.
  For this, we shall use the so-called flow interchange method \cite[Theorem 3.2]{MMS},
  where the minimizer $\mu_\tau^n$ is perturbed along the $\alpha$-flow $\flow^\aux_{(\cdot)}$ of the auxiliary functional $\aux$.
\begin{lemma}[Flow Interchange]
  \label{lem:flowinterchange}
  Let $\fnc,\aux: \propo \rightarrow \mathbb{R}\cup\{ + \infty \}$ be
  two proper, lower semicontinuous functionals with $\Dom{\fnc} \subseteq \Dom{\aux}$.
  Assume further that $\aux$ produces an $\alpha$-flow $\flow^{\aux}_{(\cdot)}$.
  Let $\mu^*$ be a global minimizer of the following Yosida-penalization of $\fnc$,
  \begin{align}\label{eq:yosida}
    \mu\mapsto\frac1{2\tau}\wasser{\mu}{\bar\mu} + \fnc(\mu),
  \end{align}
  where $\bar\mu$ is given.
  Then
  \begin{align}
    \label{eq:flowinterchange}
    \limsup_{\s\downarrow0}\frac{\fnc(\mu^*)-\fnc(\flow^\aux_\s(\mu^*))}{\s}
    \le \frac{\aux(\bar\mu)-\aux(\mu^*)}\tau
    -\frac\alpha{2\tau}\wasser{\mu^*}{\bar\mu}.
  \end{align}
\end{lemma}
  \begin{proof}[Sketch of proof]
    On the one hand, since $\mu^*$ is the minimizer in \eqref{eq:yosida},
    we have for each $\s>0$ that
    \begin{align}
      \label{eq:fi-help1}
      \frac1{\s}\big[\fnc(\mu^*)-\fnc(\flow^\aux_\s(\mu^*))\big]
      \le \frac1{2\tau\s}\big[\wass^2(\flow^\aux_\s(\mu^*),\bar\mu)-\wass^2(\mu^*,\bar\mu)\big].
    \end{align}
    On the other hand, by the EVI \eqref{eq:evi} at $s=0$,
    \begin{align}
      \label{eq:fi-help2}
      \limsup_{\s\downarrow0}\frac1{2\tau\s}\big[\wass^2(\flow^\aux_\s(\mu^*),\bar\mu)-\wass^2(\mu^*,\bar\mu)\big]
      \le \frac1{\tau} \big[\aux(\bar\mu)-\aux(\mu^*)\big]
      -\frac\alpha{2\tau}\wasser{\mu^*}{\bar\mu}.
    \end{align}
    Combining \eqref{eq:fi-help1} with \eqref{eq:fi-help2} yields \eqref{eq:flowinterchange}.
  \end{proof}
  \begin{remark}
    The left-hand side in \eqref{eq:flowinterchange} is an approximation of $\fnc$'s dissipation along $\aux$'s flow.
    At least on a formal level, one expects that it is also an approximation to $\aux$'s dissipation along $\fnc$'s flow,
    i.e., a time-discrete analogue of $-\frac{\dn}{\dd t}\aux(\mu_t)$.
    Indeed, in a corresponding smooth situation,
    with a map $x_{(\cdot,\cdot)}:\setR\times\setR\to\setR^n$,
    that is a gradient flow with respect to two ``time'' parameters $s$ and $t$,
    \begin{align*}
      \partial_tx_{(s,t)}=-\nabla U(x_{(s,t)}),
      \quad
      \partial_sx_{(s,t)}=-\nabla V(x_{(s,t)}),
    \end{align*}
    for smooth functions $U,V:\setR^n\to\setR$, we have
    \begin{align*}
      -\frac{\dn}{\dd t}V(x_{s,t}) = \nabla U(x_{(s,t)})\cdot\nabla V(x_{(s,t)}) = -\frac{\dn}{\dd s}U(x_{s,t}).
    \end{align*}
    In the non-smooth situation at hand, the two dissipations might not be identical,
    but typically, one can be controlled in terms of the other.
  \end{remark}

        
\section{Existence of Solutions}
\label{sct:existence}

\subsection{Properties of the energy functional}
We begin by giving a proper definition of the energy functional.
\begin{definition}
  The energy functional $\nrgm_\lambda:\propo\to\setR_{\ge0}\cup\{+\infty\}$ is defined as follows:
  if $\mu=u\leb$ is absolutely continuous with $\sqrt u\in W^{2,2}(\setR^d)$,
  then
  \begin{align}
    \label{eq:nrgm}
    \nrgm_\lambda(\mu):=
    2\intrn \big\|\nabla^2\sqrt u-4\nabla\sqrt[4]u\otimes\nabla\sqrt[4]u\big\|^2\dd x
    +\lambda^3\intrn|x|^2\dd\mu(x);
  \end{align}
  otherwise, $\nrgm_\lambda(\mu):=+\infty$.
\end{definition}
\begin{remark}
	\label{rem:nrgm}
  Several comments are in order.
  \begin{itemize}
  \item For well-definedness of the right-hand side in \eqref{eq:nrgm},
    we implicitly use the fact that $\sqrt u\in W^{2,2}(\setR^d)$ implies $\sqrt[4]u\in W^{1,4}(\setR^d)$,
    see e.g. \cite[Th\'{e}or\'{e}me 1(ii)]{LiVi}.
    Actually, there is a constant $C$ such that $\|\nabla\sqrt[4]u\|_{L^4}^2\le C\|\nabla^2\sqrt u\|_{L^2}$,
    and thus,
    \begin{align}
      \label{eq:nrgabove}
      \nrgm_0(\mu) \le C'\|\nabla^2\sqrt u\|_{L^2}^2.
    \end{align}
  \item The condition $\sqrt u\in W^{2,2}(\setR^d)$
    is an essential part of the definition.
    Note that in principle, cancellation effects inside the norm under the integral in \eqref{eq:nrgm}
    could lead to a finite integral value despite $\sqrt u\notin W^{2,2}(\setR^d)$.
    Our choice of $\nrgm_\lambda$'s domain is justified by Propositions \ref{prp:savarebnd} and \ref{prp:savarelsc} below.
  \item If $\mu$'s density $u$ is positive everywhere,
    then one has
    \[ \nabla^2\sqrt u-4\nabla\sqrt[4]u\otimes\nabla\sqrt[4]u=\frac12\sqrt u\,\nabla^2\log u, \]
    and so $\nrgm_\lambda(\mu)=\nrg_\lambda(u)$ with $\nrg_\lambda$ defined in \eqref{eq:HFE}.
  \item The authors of \cite{GST} consider a slightly different  version $\mathfrak{K}_{-1}(\cdot|\leb)$ of $\nrg_0$:
    assuming $\mu=u\leb$ satisfies the weaker condition $\sqrt u\in W^{2,2}_\loc(\setR^d)$, they set
    \begin{align}
      \label{eq:savareI}
      \mathfrak{K}_{-1}(\mu|\leb)
      = \intrn \left\|\frac{\sqrt u\nabla^2\sqrt u-\nabla\sqrt u\otimes\nabla\sqrt u}{\sqrt u}\right\|^2\dd\mu(x).
    \end{align}
    Since the zero set of $u$ is clearly $\mu$-negligible, the integrand is $\mu$-a.e. well-defined.
    For $\sqrt u\in W^{2,2}(\setR^d)$, one has $\mathfrak{K}_{-1}(\mu|\leb)=\nrgm_0(\mu)$;
    this is a consequence of the fact that the derivatives of Sobolev functions are zero $\leb$-a.e. on their level sets,
    i.e., the integrand in \eqref{eq:nrgm} vanishes a.e. on $\{u=0\}$.
    The representation in \eqref{eq:savareI} is the relevant one in the proof of lower semi-continuity,
    see Proposition \ref{prp:savarelsc} below.
    For our later needs, the representation \eqref{eq:nrgm} is better suited.
  \end{itemize}
\end{remark}
The energy $\nrgm_\lambda$ defined above is part of a family of second-order functionals
that have been studied in \cite[Section 3]{GST} in connection with the DLSS equation \eqref{eq:DLSS}.
We recall and adapt two results here.
\begin{proposition}[adapted from Corollary 3.2 in \cite{GST}]
  \label{prp:savarebnd}
  There is a constant $C$, only depending on the dimension $d$,
  such that for all absolutely continuous $\mu=u\leb\in\propo$:
  \begin{align}
    \label{eq:savarebnd}
    \intrn\big(\big\|\nabla^2\sqrt u\big\|^2+\big|\nabla\sqrt[4]u\big|^4\big)\dd x\le C\nrgm_\lambda(\mu).
  \end{align}
\end{proposition}
It is important to remark here that \cite[Corollary 3.2]{GST} indeed applies to the case $\Omega=\setR^d$:
the derivation of \eqref{eq:savarebnd} only involves an integration by parts
with the vector field $\velo=|\nabla\sqrt[4]u|^2\nabla\sqrt u$,
which is integrable on $\setR^d$ since $\sqrt u\in W^{2,2}(\setR^d)$ and $\sqrt[4]u\in W^{1,4}(\setR^d)$.
We refer to Lemma \ref{lem:Gauss} and to the subsequent discussion in the Appendix.
\begin{proposition}[adapted from Corollary 3.4 of \cite{GST}]
  \label{prp:savarelsc}
  The functional $\nrgm_\lambda$ is sequentially lower semi-continuous
  with respect to narrow convergence.
\end{proposition}
\begin{proof}
  In \cite[Corollary 3.4]{GST}, the lower semi-continuity
  of $\mathfrak{K}_{-1}(\cdot|\leb)$ recalled in \eqref{eq:savareI} above is shown.
  The only difference between $\mathfrak{K}_{-1}(\mu|\leb)$ and $\nrgm_0(\mu)$ is that
  the former is defined by the integral value (possibly $+\infty$)
  for all $\mu=u\leb$ with $\sqrt u\in W^{2,2}_\loc(\setR^d)$,
  and the latter is $+\infty$ unless $\sqrt u\in W^{2,2}(\setR^d)$.

  The ``stability'' of $\nrgm_0$'s restricted domain follows directly from Proposition \ref{prp:savarebnd} above:
  if $(\mu_n)$ converges narrowly to $\mu$ and has $\sup_n\nrgm_\lambda(\mu_n)<\infty$,
  then $(\sqrt {u_n})$ is bounded in $W^{2,2}(\setR^d)$ thanks to \eqref{eq:savarebnd}.
  Consequently, $(\sqrt{u_n})$ converges strongly to a limit $\sqrt u$ in $L^2(\setR^d)$,
  which in turn implies that $\mu=u\leb$ is absolutely continuous.
  And boundedness of $(\sqrt{u_n})$ in $W^{2,2}(\setR^d)$ further implies that also $\sqrt u\in W^{2,2}(\setR^d)$.
\end{proof}


\subsection{Minimizing movements}
\label{sct:mm}
Let a time step size $\tau>0$ be fixed.
For a given $\nu\in\propo$,
define the Yosida-penalized energy functional $\nrgm_{\lambda,\tau}(\cdot;\nu)$ by
\begin{align*}
  \nrgm_{\lambda,\tau}(\mu;\nu) = \frac1{2\tau}\wass(\mu,\nu)^2 + \nrgm_\lambda(\mu).
\end{align*}
\begin{lemma}
  \label{lem:mm}
  For each $\nu\in\propo$, there exists a global minimizer $\mu^*\in\propo$ of $\nrgm_{\lambda,\tau}(\cdot;\nu)$. 
\end{lemma}
\begin{proof}
  This is a standard argument from the calculus of variations.
  Observe that the functional $\mu\mapsto \nrgm_{\lambda,\tau}(\mu;\nu)$ has the following properties:
  \begin{itemize}
  \item It is bounded from below (in fact: non-negative),
    and is not identically $+\infty$ (since it has a finite value
    for any absolutely continuous $\mu=u\leb\in\propo$ with $\sqrt u\in W^{2,2}(\setR^d)$).
    Hence, it has a finite infimum.
  \item It is coercive.
    Indeed, by non-negativity of $\nrgm_\lambda$ and the properties of $\wass$,
    one easily shows that $\nrgm_{\lambda,\tau}(\mu;\nu)\ge \frac1{4\tau}\mom(\mu) - C$
    with a $C$ that is expressible in terms of $\nu$ and $\tau$.
    Hence, sublevels are tight and thus pre-compact with respect to narrow convergence.
  \item It is sequentially lower semi-continuous with respect to narrow convergence,
    see Proposition \ref{prp:savarelsc} above.
  \end{itemize}
  The existence of a minimizer now follows by standard arguments.
\end{proof}
As a consequence of Lemma \ref{lem:mm},
the \emph{minimizing movement scheme} for $\nrgm$ is well-defined
for every initial condition $\mu_0\in\propo$.
That is, starting from $\mu_\tau^0:=\mu_0$, 
one can define a sequence $(\mu_\tau^n)_{n=0}^\infty$ inductively 
by choosing for $\mu_\tau^n$ as a minimizer of $\nrgm_\tau(\cdot;\mu_\tau^{n-1})$ for $n=1,2,\ldots$
For notational convenience,
we also introduce the usual time-discrete ``interpolation'' $\bar\mu_\tau:\setR_{\ge0}\to\propo$ by
\begin{align*}
  \bar\mu_\tau(t) = \mu_\tau^n\quad\text{for all $t\in((n-1)\tau,n\tau]$}
\end{align*}
for $n=1,2,\ldots$, and $\bar\mu_\tau(0)=\mu_0$.
The sequence $(\mu_\tau^n)_{n=0}^\infty$ and its interpolation $\bar\mu_\tau$ satisfy a variety of energy estimates,
that directly follow from the construction via minimization.
\begin{lemma}[Basic discrete estimates]
  For every $N=1,2,\ldots$, the energy $\nrgm_\lambda(\mu_\tau^N)$ is finite,
  and moreover,
  \begin{align}
    \label{eq:ce-mono}
    & \nrgm_\lambda(\mu_\tau^N) \leq \nrgm_\lambda(\mu_\tau^{N-1}) \leq \nrgm_\lambda(\mu_0).  \\
    \label{eq:ce-sum}
    & \frac\tau2 \sum_{n = N}^{\infty} \left(\frac{\wassernoq{\mu_\tau^{n+1}}{\mu_\tau^{n}}}\tau\right)^2
      \leq \nrgm_\lambda(\mu_\tau^{N}).
  \end{align}
  And further, for all $s,t\ge N\tau$,
  \begin{align}
    \label{eq:ce-holder}
    \wasser{\bar\mu_\tau(t)}{\bar\mu_\tau(s)} \leq 2\nrgm_\lambda(\mu_\tau^N)\,\max\{ \tau, |t-s| \}.
  \end{align}
\end{lemma}
The derivation of these estimates is a standard procedure, see e.g. \cite{AGS}.

\subsection{Discrete equation}
\label{sct:deq}
In this section, we derive a time-discrete surrogate of \eqref{eq:QDD6intro}
that is satisfied by the discrete approximation $(\mu_\tau^n)$.
Following the seminal idea from \cite{JKO,GST},
we produce a time-discrete, very weak formulation of \eqref{eq:QDD6intro}
--- eventually leading to \eqref{eq:weak} --- 
by performing an inner variation of the minimizer $\mu_\tau^n$ of $\nrgm_\tau(\cdot;\mu_\tau^{n-1})$.
That weak formulation is well-defined under the hypothesis that $\mu_\tau^n\in\Dom{\nrgm_\lambda}$,
which follows trivially by the construction.

More specifically,
let a smooth and compactly supported function $\varphi\in C^\infty_c(\setR^d)$ be given,
and define the associated $\setR^d$-gradient flow $X_{(\cdot)}: \setR\times\setR^d\to\setR^d$
as solution to the ODE initial value problem
\begin{equation}
  \label{eq:ourODE}
  \frac{\dn}{\dd\s} X_\s = -\nabla\varphi\circ X_\s, \quad X_0 = \id. 
\end{equation}
For a given absolutely continuous measure $\mu=u\leb \in \Dom{\nrgm_\lambda}$,
we consider the deformations $\mu_\s:=X_\s\#\mu\in\propo$.
These satisfy the continuity equation along the vector field $- \nabla\varphi$, i.e.,
\begin{align}
  \label{eq:cont}
  \partial_\s\mu_\s = \dv(\mu_\s\nabla\varphi) .
\end{align}
In Lemma \ref{lem:1stvariation} below,
the $\s$-derivative of $\nrgm_\lambda(\mu_\s)$ at $\s=0$ is given explicitly.
In view of \eqref{eq:cont}, we obtain formally
--- that is, in case that $u$ is positive and smooth ---
that
\begin{align*}
  \frac{\dn}{\dd\s}\bigg|_{\s=0}\nrgm_\lambda(\mu_\s)
  &= \intrn \left( \frac{\delta\nrg_0}{\delta u} + \lambda^3 |x|^2 \right)\dv(u\nabla\varphi) \dd x\\
  &= \intrn \varphi\, \dv\big(u\nabla\Phi(u)\big)\dd x -2\lambda^3\intrn x\cdot\nabla\varphi u \dd x.
\end{align*}
In other words, calculating the $\s$-derivative of $\nrgm_\lambda(\mu_\s)$ at $\s=0$
amounts to calculating a special form of the right-hand side in \eqref{eq:QDD6intro},
``tested'' against $\varphi$.
This philosophy --- which was the founding idea behind the flow interchange lemma, see \eqref{eq:flowinterchange} ---
is made rigorous in Lemma \ref{lem:discweakform} further below.
\begin{lemma}[First variation]
  \label{lem:1stvariation}
  Let $\mu=u\leb\in \Dom{\nrgm_\lambda}$
  and define accordingly $s:=\sqrt{u}\in W^{2,2}(\setR^d)$ and $z:=\sqrt[4]{u}\in W^{1,4}(\setR^d)$.
  Given $\varphi\in C^\infty_c(\setR^d)$, define the flow $X_{(\cdot)}$ as in \eqref{eq:ourODE} above.
  Then the map $\s \mapsto \nrgm_\lambda(X_\s\#\mu)$ is differentiable in $\s= 0$,
  with derivative:
  \begin{align}
    \label{eq:EL}
    \frac{\dn}{\dd\s}\bigg|_{\s=0}\nrgm_\lambda(X_\s\#\mu)
    = -\N{u}{\varphi}
    -2\lambda^3\intrn x\cdot\nabla\varphi\dd\mu,
  \end{align}
  with $\N{u}{\phi}$ given in \eqref{eq:N}.
\end{lemma}
\begin{proof}
  First, let us assume that $\lambda=0$;
  the minor modifications for $\lambda>0$ are described at the end of the proof.
  Introduce $u_\s$ as the density of $X_\s\#\mu$,
  and accordingly $s_\s=\sqrt{u_\s}$ and $z_\s=\sqrt[4]{u_\s}$.
  For later reference, observe that
  \begin{align*}
    X_\s = \id - \s\nabla\varphi + O(\s^2),
    \quad
    \dff X_\s = \eins - \s\dff\nabla\varphi + O(\s^2),
    \quad
    \dff^2X_\s = -\s\dff^2\nabla\varphi + O(\s^2).
  \end{align*}
  Next, define the volume distortion $V_\s:=\det(\dff X_\s)$,
  which is a positive and smooth function,
  and observe that
  \begin{align*}
    V_\s = 1 - \s\Delta\varphi + O(\s^2),
    \quad
    \dff V_\s = -\s \dff\Delta\varphi + O(\s^2),
    \quad
    \dff^2V_\s = -\s\dff^2\Delta\varphi + O(\s^2).
  \end{align*}
  By the change of variables formula, we have with $x=X_\s(y)$,
  \begin{equation}
    \label{eq:nrgs}
    \begin{split}
      \nrgm_0(X_\s\#\mu) 
      &= 2 \intrn \|\nabla^2 s_\s-4\nabla z_\s\otimes\nabla z_\s\|^2\dd x \\
      &= 2 \intrn
      \big\|V_\s^{1/2}\nabla^2 s_\s\circ X_\s-4(V_\s^{1/4}\nabla z_\s\circ X_\s)\otimes(V_\s^{1/4}\nabla z_\s\circ X_\s)\big\|^2
      \dd y
    \end{split}
  \end{equation}
  We shall now express the spatial derivatives of $s_\s$ and $z_\s$
  in terms of the respective derivatives of $s$ and $z$.
  For the next calculations, which involve a repeated application of the chain rule,
  we use instead of gradients and Hessians
  the less intuitive but more consistent notations with total derivatives $\dff$ that produce row vectors.

  Recall the effect of the push-forward on densities:
  \begin{equation}
    \label{eq:recallpush}
    u_\s= \frac{u}{V_\s}\circ X_\s^{-1}.
  \end{equation}
  Hence we have:
  \begin{align}
    \label{eq:sztrafo}
    s_\s = (V_\s^{-1/2}s)\circ X_\s^{-1},
    \quad
    z_\s = (V_\s^{-1/4}z)\circ X_\s^{-1}.
  \end{align}
  For the first derivatives, we thus obtain
  \begin{align*}
    \dff s_\s &= \left\{\left(V_\s^{-1/2}\dff s-\frac12V_\s^{-3/2}s\dff V_\s\right)(\dff X_\s)^{-1}\right\}\circ X_\s^{-1}, \\
    \dff z_\s &= \left\{\left(V_\s^{-1/4}\dff z-\frac14V_\s^{-5/4}z\dff V_\s\right)(\dff X_\s)^{-1}\right\}\circ X_\s^{-1},
  \end{align*}
  and for the second derivative of $s_\s$,
  \begin{align*}
    \dff^2s_\s
    &= \bigg\{
      \bigg[V_\s^{-1/2}\dff^2s-V_\s^{-3/2}\dff s\otimes_s\dff V_\s
      -\frac12V_\s^{-3/2}s\dff^2V_\s + \frac34V_\s^{-5/2}s\dff V_\s\otimes\dff V_\s\\
    &\qquad 
      -\left(V_\s^{-1/2}\dff s-\frac12V_\s^{-3/2}s\dff V_\s\right)(\dff X_\s)^{-1}\dff^2X_\s
      \bigg]
      :\big((\dff X_\s)^{-1}\otimes (\dff X_\s)^{-1}\big)
      \Big\}\circ X_\s^{-1}.                 
  \end{align*}
  The expression on the right hand side calls for some explanation:
  the part in the square brackets is a symmetric bilinear form;
  and when $\dff^2s_\s\circ X_\s$ is applied to two vectors $\xi,\eta\in\setR^d$,
  then that bilinear form is evaluated on the vectors $(\dff X_\s)^{-1}\xi,(\dff X_\s)^{-1}\eta$.
  
  Since all the $\s$-dependence on the right-hand side is now in $X_\s$ and $V_\s$,
  these expressions are obviously differentiable in $\s$,
  with derivatives
  \begin{align*}
    \frac{\dn}{\dd\s}\bigg|_{\s=0}\big(V_\s^{1/4}\dff z_\s\circ X_\s\big)
    &=\frac{\dn}{\dd\s}\bigg|_{\s=0}
      \left\{\left(\dff z-\frac14V_\s^{-1}z\dff V_\s\right)(\dff X_\s)^{-1}\right\}
    &= \dff z\dff\nabla\varphi + \frac14z\dff\Delta\varphi
  \end{align*}
  and
  \begin{align*}
    &\frac{\dn}{\dd\s}\bigg|_{\s=0}\big(V_\s^{1/2}\dff^2s_\s\circ X_\s\big) \\
    &= \frac{\dn}{\dd\s}\bigg|_{\s=0}
    \bigg\{
      \bigg[\dff^2s-V_\s^{-1}\dff s\otimes_s\dff V_\s
      -\frac12V_\s^{-1}s\dff^2V_\s + \frac34V_\s^{-2}s\dff V_\s\otimes\dff V_\s \\
    &\qquad -\left(\dff s-\frac12V_\s^{-1}s\dff V_\s\right)(\dff X_\s)^{-1}\dff^2X_\s
    \bigg]
      :\big((\dff X_\s)^{-1}\otimes (\dff X_\s)^{-1}\big)
      \bigg\}
    \\
    &= 2\dff^2s\dff\nabla\varphi +\dff s\otimes_s\dff\Delta\varphi
    +\frac12s\dff^2\Delta\varphi +\dff s\dff^2\nabla\varphi.
  \end{align*}
  We are now in the position to calculate the $\s$-derivative of $\nrgm(X_\s\#\mu)$ at $\s=0$
  by differentiating in \eqref{eq:nrgs} directly under the integral.
  Recall the abbreviation $\ell=\nabla^2s-4\nabla z\otimes\nabla z$, which is a symmetric $d\times d$-matrix.
  For the derivative of the integrand we obtain:
  \begin{align*}
    &\frac12\frac{\dn}{\dd\s}\bigg|_{\s=0}
    \big\|V_\s^{1/2}\nabla^2 s_\s\circ X_\s-4(V_\s^{1/4}\nabla z_\s\circ X_\s)\otimes(V_\s^{1/4}\nabla z_\s\circ X_\s)\big\|^2 \\
    &= \left(\frac{\dn}{\dd\s}\bigg|_{\s=0}\big(V_\s^{1/2}\nabla^2s_\s\circ X_\s\big)\right):\ell
      - 8\nabla z\cdot\ell\cdot\left(\frac{\dn}{\dd\s}\bigg|_{\s=0}\big(V_\s^{1/4}\nabla z_\s\circ X_\s\big)\right) \\
    &= 2\nabla^2\varphi:(\ell\cdot\nabla^2s) + \nabla\Delta\varphi\cdot\ell\cdot\nabla s
      +\frac12s\nabla^2\Delta\varphi:\ell + \nabla s\cdot\nabla^3\varphi:\ell \\
    &\qquad -8\nabla^2\varphi: \big(\ell\cdot(\nabla z\otimes\nabla z)\big)
      -2z\nabla\Delta\varphi\cdot\ell\cdot\nabla z \\
    &=2\nabla^2\varphi:(\ell^2) + \frac12(s\nabla^2\Delta\varphi+2 \nabla s\cdot\nabla^3\varphi):\ell,
  \end{align*}
  where we have used that $\nabla s=2z\nabla z$ to cancel the two terms multiplying $\nabla\Delta\varphi$.
  Integration of this equality with respect to $x$ yields \eqref{eq:EL} with $\lambda=0$.

  When $\lambda>0$, we calculate
  \begin{align*}
    \frac{\dn}{\dd\s}\bigg|_{\s=0}\nrgm_\lambda(X_\s\#\mu)
    = \frac{\dn}{\dd\s}\bigg|_{\s=0}\nrgm_0(X_\s\#\mu) 
    +\lambda^3\frac{\dn}{\dd\s}\bigg|_{\s=0}\intrn|x|^2\dd X_\s\#\mu(x).
  \end{align*}
  Since by definition of the push-forward,
  \begin{align*}
    \intrn|x|^2\dd(X_\s\#\mu) = \intrn |X_\s(y)|^2\dd\mu(y),
  \end{align*}
  we directly obtain
  \begin{align*}
    \frac{\dn}{\dd\s}\bigg|_{\s=0}\intrn|X_\s(y)|^2\dd X_\s\#\mu(y)
    = \intrn 2X_0(y)\cdot \frac{\dn}{\dd\s}\bigg|_{\s=0}X_\s(y)\dd\mu(y)
    = -2\intrn y\cdot\nabla\varphi(y)\dd\mu(y),
  \end{align*}
  which provides the additional term in \eqref{eq:EL}.
\end{proof}
\begin{lemma}
  \label{lem:discweakform}
  For any test function $\phi \in C^\infty_c(\setR^d)$, there is a constant $\alpha$
  such that for each $n\ge1$,
  the measures $\mu_\tau^n=u_\tau^n\leb$ and $\mu_\tau^{n-1}=u_\tau^{n-1}\leb$
  satisfy the following time discrete version of \eqref{eq:weak}:
  \begin{equation}
    \label{eq:discweak}
    \left| \intrn \phi \frac{u_\tau^n - u_\tau^{n-1}}\tau\dd x + \N{u^n_\tau}{\phi} + 2\lambda^3\intrn x\cdot\nabla\varphi u_\tau^n\dd x\right|
      \leq \frac{\alpha\tau}{2} \left(\frac{\wassernoq{u^n_\tau}{u^{n-1}_\tau}}\tau\right)^2.
  \end{equation}
\end{lemma}
\begin{proof}
  Choose $\alpha>0$ such that $-\alpha \one \leq \nabla^2 \phi \leq \alpha \one$.
  According to Example \ref{xmp:flow},
  the solution $\mu_\s=X_\s\#\mu_\tau^n$ to the continuity equation \eqref{eq:cont}
  follows an $(-\alpha)$-flow for the potential energy $\aux(\mu)=\intrn\varphi(x)\dd\mu(x)$ emerging from $\mu_\tau^n$.
  As a consequence of the flow interchange estimate \eqref{eq:flowinterchange},
  \begin{align*}
    -\frac{\dn}{\dd\s}\bigg|_{\s=0}\nrgm_\lambda(X_\s\#\mu_\tau^n)
    &= \limsup_{\s\downarrow0} \frac{\nrgm_\lambda(\mu_\tau^n)-\nrgm_\lambda(X_\s\#\mu_\tau^n)}{\s} \\
    &\le \frac{\aux(\mu_\tau^{n-1})-\aux(\mu_\tau^n)}{\tau} + \frac{\alpha}{2\tau}\wasser{\mu_\tau^n}{\mu_\tau^{n-1}} \\
    &\quad = -\intrn \varphi\frac{u_\tau^{n}-u_\tau^{n-1}}{\tau}\dd x
      + \frac{\alpha}{2\tau}\wasser{\mu_\tau^n}{\mu_\tau^{n-1}}.
  \end{align*}
  Substitution of \eqref{eq:EL} for the $\s$-derivative of $\nrgm_\lambda$ yields
  \begin{align*}
    \intrn \phi \frac{u_\tau^n - u_\tau^{n-1}}\tau\dd x + \N{u^n_\tau}{\phi} + 2\lambda^3\intrn x\cdot\nabla\varphi u_\tau^n\dd x
    \le  \frac{\alpha}{2\tau}\wasser{\mu_\tau^n}{\mu_\tau^{n-1}}.
  \end{align*}
  Replacing $\phi$ by $-\phi$ produces the same inequality (also with the same value of $\alpha$)
  but with an overall minus on the left-hand side.
  Combination of these two estimates leads to \eqref{eq:discweak}.
\end{proof}


\subsection{Additional a priori estimates}
The next step is to derive a time-discrete version of the a priori estimate \eqref{eq:apriori}.
\begin{proposition}
  \label{prp:apriori}
  The sequence $(\mu_\tau^n)$ of time-discrete approximations $\mu_\tau^n=u_\tau^n\leb$ constructed above
  satisfies at each $n=1,2,\ldots$
  \begin{equation}
    \label{eq:bpriori}
    \kappa \intrn\big(\interleave\nabla^3 \sqrt{u_\tau^n}\interleave^2
    +|\nabla \sqrt[6]{u_\tau^n}|^6 \big)\dd x
    \leq \frac{\ent_0(u_\tau^{n-1}) - \ent_0(u_\tau^n)}\tau +2d\lambda^3,
  \end{equation}
  where $\kappa > 0$ is a constant that is expressible in terms of the dimension $d$ alone.
\end{proposition}
The proof of the Lemma builds on the analogous result derived in \cite{BJM}
for a --- more hands-on --- time-discrete approximation of solutions to \eqref{eq:QDD6intro}
with \emph{periodic boundary conditions}.
The formal calculations there are identical to the ones that lead to \eqref{eq:bpriori}.
The rigorous justification is more difficult:
first because the time steps in \cite{BJM} have a higher degree of spatial regularity than the Wasserstein approximants here;
and second, because the periodic boundary conditions make any discussion of boundary terms related to integration by parts unnecessary.

We shall perform several approximations before we can apply the formal calculations from \cite{BJM}.
For notational convenience, we assume $\lambda=0$ for the moment;
the minor modifications for $\lambda>0$ are summarized at the end of the proof of Proposition \ref{prp:apriori}.

The first ingredient in our approximation procedure is the following regularization of the energy functional:
for each $\eps>0$, define
\begin{align*}
  \nrg^\eps(u) : = 2\intrn \big\|\nabla^2\sqrt{u+\eps}-4\nabla\sqrt[4]{u+\eps}\otimes\nabla\sqrt[4]{u+\eps}\|^2\dd x.
\end{align*}
for $u\in L^1(\setR^d)$ with $\sqrt u\in W^{2,2}(\setR^d)$.
\begin{lemma}
  \label{lem:nrgeps}
  Assume that $u\leb \in \Dom{\nrgm_0}$.
  Then also $\nrg^\eps(u)<\infty$, and $\lim_{\eps\downarrow0}\nrg^\eps(u)=\nrgm_0(u\leb)$.
\end{lemma}
\begin{proof}
  By definition of $\nrgm_0$, we have that $\sqrt u\in W^{2,2}(\setR^d)$ and $\sqrt[4]u\in W^{1,4}(\setR^d)$.
  According to the chain rule for concatenation of $u$ with
  the smooth functions $r\mapsto \sqrt{r+\eps}$ and $r\mapsto\sqrt[4]{r+\eps}$
  that are sublinear with globally bounded first and second derivatives,
  also $\sqrt{u+\eps}\in W^{2,2}_\loc(\setR^d)$ and $\sqrt[4]{u+\eps}\in W^{1,4}_\loc(\setR^d)$,
  and
  \begin{align*}
    \nabla^2\sqrt{u+\eps}
    &= \left[\frac u{u+\eps}\right]^{1/2}
      \left(\nabla^2\sqrt u +4\frac{\eps}{u+\eps}\nabla\sqrt[4]u\otimes\nabla\sqrt[4]u\right), \\
    \nabla\sqrt[4]{u+\eps}
    &= \left[\frac u{u+\eps}\right]^{3/4}\nabla\sqrt[4]u.
  \end{align*}
  In particular, $\nabla^2\sqrt{u+\eps}\in L^2(\setR^d)$ and $\nabla\sqrt[4]{u+\eps}\in L^4 (\setR^d)$,
  and so $\nrg^\eps(u)<\infty$.
  Further, the quotients $u/(u+\eps)$ and $\eps/(u+\eps)$ are bounded by one,
  and converge to one and to zero respectively, in measure $u\leb$-a.e.
  It follows by dominated convergence that
  \begin{align*}
    \nabla^2\sqrt{u+\eps} \to \nabla^2\sqrt u\quad\text{in $L^2(\setR^d)$},
    \qquad
    \nabla\sqrt[4]{u+\eps} \to \nabla\sqrt[4]u\quad\text{in $L^4(\setR^d)$},
  \end{align*}
  and therefore also $\nrg^\eps(u)\to\nrgm_0(u\leb)$.
\end{proof}
We are now going to study certain properties of $\nrg^\eps$ along solutions of the heat flow.
More specifically, let some probability density $u\in L^1(\setR^d)$ be given,
and consider for each $r>0$:
\begin{align}
  \label{eq:heatbyconv}
  w_r := K_r\ast u \quad \text{for} \quad K_r(z)=(4\pi r)^{-d/2}\exp\left(-\frac{|z|^2}{4r}\right).
\end{align}
Note that $w_r$ is the unique solution to the initial value problem
\begin{align*}
  \partial_rw_r=\Delta w_r, \quad w_0=u.
\end{align*}
Our first observation is that $r\mapsto\nrg^\eps(w_r)$ has the expected limit for $r\downarrow0$.
\begin{lemma}
  \label{lem:epscont}
  Assume that $u\leb \in \Dom{\nrgm_0}$,
  define $w_r$ as in \eqref{eq:heatbyconv} above.
  Then $\lim_{r\downarrow0}\nrg^\eps(w_r)=\nrg^\eps(u)$.
\end{lemma}
\begin{proof}
  We present a proof that heavily uses the dimensionality restriction $d\le3$.
  
  By definition of $\nrgm_0$, we have that $\sqrt u\in W^{2,2}(\setR^d)$ and $\sqrt[4]u\in W^{1,4}(\setR^d)$.
  First, we show that
  \begin{align}
    \label{eq:uregular}
    u\in  W^{2,2}(\setR^d)\cap W^{1,4}(\setR^d).    
  \end{align}
  For this, we use that since $d\le 3$, one has $W^{2,2}(\setR^d)\hookrightarrow L^\infty(\setR^d)$,
  and so $u$ is globally bounded.
  Now, by the chain rule,
  \begin{align*}
    \nabla u &= \nabla\big(\sqrt[4]{u}^4\big) = 4\sqrt[4]{u}^3\nabla\sqrt[4]{u}, \\
    \nabla^2u&= \nabla^2\big(\sqrt u^2\big) = 2\sqrt u\nabla^2\sqrt u + 2\nabla\sqrt u\otimes\nabla\sqrt u
               = 2 \sqrt u\nabla^2\sqrt u + 8\sqrt[4]u^2\nabla\sqrt[4]u\otimes\nabla\sqrt[4]u,
  \end{align*}
  which shows $\nabla u\in L^4(\setR^d)$ and $\nabla^2u\in L^2(\setR^d)$.
  Interpolation with $u\in L^1(\setR^d)$ yields \eqref{eq:uregular}.

  By the representation \eqref{eq:heatbyconv} of $w_r$ as convolution with $K_r$,
  and thanks to Young's inequality, it follows from \eqref{eq:uregular} that
  \begin{align}
    \label{eq:wratzero}
    w_r\to u \quad\text{in $W^{2,2}(\setR^d)\cap W^{1,4}(\setR^d)$}.
  \end{align}
  Recall that $w_r$ is smooth and positive.
  It follows by the chain rule that:
  \begin{align*}
    \nabla\sqrt[4]{w_r+\eps} &=  \frac14(w_r+\eps)^{-3/4}\nabla w_r, \\
    \nabla^2\sqrt{w_r+\eps} &= \frac12(w_r+\eps)^{-1/2}\nabla^2w_r -\frac14(w_r+\eps)^{-3/2}\nabla w_r\otimes\nabla w_r.
  \end{align*}
  Combining \eqref{eq:wratzero} with the fact that, for any $p>0$,
  the function $(w_r+\eps)^{-p}$ is bounded by $\eps^{-p}$,
  and converges in measure to $(u+\eps)^{-p}$ as $r\downarrow0$,
  we conclude that
  \begin{align*}
    \nabla\sqrt[4]{w_r+\eps}
    &\to \frac14(u+\eps)^{-3/4}\nabla u = \nabla\sqrt[4]{u+\eps}
      \quad \text{in $L^4(\setR^d)$}, \\
    \nabla^2\sqrt{w_r+\eps}
    &\to \frac12(u+\eps)^{-1/2}\nabla^2u -\frac14(u+\eps)^{-3/2}\nabla u\otimes\nabla u = \nabla^2\sqrt{u+\eps}
      \quad\text{in $L^2(\setR^d)$},
  \end{align*}
  and so we have that
  \begin{align*}
    \nabla^2\sqrt{w_r+\eps}-4\nabla\sqrt[4]{w_r+\eps}\otimes\nabla\sqrt[4]{w_r+\eps}
    \to
    \nabla^2\sqrt{u+\eps}-4\nabla\sqrt[4]{u+\eps}\otimes\nabla\sqrt[4]{u+\eps}
    \quad \text{in $L^2(\setR^d)$},
  \end{align*}
  which is the claim.
\end{proof}
The next result is our core computation,
namely of the derivative of $\nrg^\eps(w_r)$ at $r>0$.
\begin{lemma}
  \label{lem:cpriori}
  Given a probability density $u$, define $w_r$ by \eqref{eq:heatbyconv}. 
  Then $r\mapsto \nrg^\eps(w_r)$ is differentiable at each $r>0$,
  with
  \begin{align}
    \label{eq:cpriori}
    -\frac{\dn}{\dd r}\nrg^\eps(w_r)
    \ge \kappa\intrn\big(\interleave\nabla^3\sqrt{w_r+\eps}\interleave^2+|\nabla\sqrt[6]{w_r+\eps}|^6\big)\dd x.
  \end{align}
\end{lemma}
\begin{proof}
  The heat kernel $K_s(z)$ from \eqref{eq:heatbyconv}
  is smooth and positive at every $s>0$ and $z\in\setR^d$,
  and all spatial derivatives $\nabla^\alpha K_s$ with arbitrary multi-index $\alpha$
  belong to any $L^p(\setR^d)$ with $p\in[1,\infty]$.
  Consequently,
  the function $\setR_{>0}\times\setR^d\ni(t;x)\mapsto w_r(x)$ is positive and smooth,
  each $w_r$ is a probability density,
  and, thanks to Young's integral inequality,
  all spatial derivatives $\nabla^\alpha w_r$ are in any $L^p(\setR^d)$.

  Next, observe that $y_r:=\log(w_r+\eps)$ is a smooth function as well,
  that satisfies
  \begin{align}
    \label{eq:heateps}
    \partial_re^{y_r}=\Delta e^{y_r}
  \end{align}
  in the classical sense.
  Moreover, despite the fact that $y_r$ itself is clearly not integrable on $\setR^d$,
  its spatial derivatives $\partial^\alpha y_r$ belong to any $L^p(\setR^d)$;
  the latter is most easily seen from the fact that each $\partial^\alpha y_r$
  can be written as a linear combination of products of terms $\partial^\beta w_r/(w_r+\eps)$,
  with suitable multi-indices $\beta$, where $1\le|\beta|\le|\alpha|$.
  This further means that also $e^{y_r}=w_r+\eps$ times any linear combination of products
  of spatial derivatives of $y_r$ are integrable.
  In particular, Gauss' theorem as stated in Lemma \ref{lem:Gauss} in the Appendix
  is applicable to vector fields built from such functions.
  
  We can now calculate the derivative of $\nrgm(w_r\leb)$,
  which we write equivalently in the form
  \begin{align*}
    \nrg^\eps(w_r) = \intrn (w_r+\eps)\big\|\nabla^2\log(w_r+\eps)\big\|^2\dd x
    = \intrn e^{y_r}\|\nabla^2y_r\|^2\dd x.
  \end{align*}  
  Thanks to the aforementioned smoothness of $y_r$
  and the admissibility of integration by parts via Lemma \ref{lem:Gauss} from the Appendix,
  we obtain --- recalling \eqref{eq:heateps}, and supressing $y$'s sub-index $r>0$ from now on ---
  \begin{equation}
    \label{eq:nrgdiss}    
    \begin{split}
      -\frac{\dn}{\dd r}\nrg^\eps(w_r)
      &= -\frac12\intrn \Delta e^y\|\nabla^2y\|^2\dd x
      - \intrn e^y \nabla^2 y:\nabla^2(e^{-y}\Delta e^y)\dd x\\
      &= \intrn \nabla e^y\cdot\nabla^3y:\nabla^2y\dd x
      - \intrn e^y \nabla^2 y:\nabla^2\big(\Delta y+|\nabla y|^2\big)\dd x \\
      &= \intrn e^y\nabla y\cdot\nabla^3y:\nabla^2y\dd x
      +\intrn \nabla\big(e^y\nabla^2 y)\vdots\nabla^3y\dd x
      - 2\intrn e^y\nabla^2y:\nabla(\nabla^2y\cdot\nabla y) \dd x \\
      &= \intrn e^y \interleave\nabla^3y\interleave^2\dd x
      - 2\intrn e^y\nabla^2 y:\big([\nabla^2y]^2\big)\dd x.
    \end{split}
  \end{equation}
  From this point on, we proceed in analogy to the proof of \cite[inequality (19)]{BJM}.
  That means, we define (ad hoc) the smooth and integrable vector field
  \begin{align*}
    \velo &= e^y\big\{\big[2|\nabla y|^2+\Delta y+5\nabla y\cdot\nabla^2y\cdot\nabla y)+5\nabla y\cdot\nabla\Delta y\big]\nabla y \\
    &\quad + \big[3|\nabla y|^2\nabla y+11\nabla\Delta y+24\nabla y\cdot\nabla^2y\big]\cdot\nabla^2y
    -\big[5\nabla y\otimes\nabla y+11\nabla^2y\big]:\nabla^3y\big\}.
  \end{align*}
  The proof of \cite[inequality (19)]{BJM} amounts to showing that
  \begin{align}
    \label{eq:monster}
    12e^y\big[\interleave\nabla^3y\interleave^2-2\nabla^2 y:\big([\nabla^2y]^2\big)\big]
    +\dv\velo
    \ge \kappa \big[2^6\interleave\nabla^3e^{y/2}\interleave^2+6^6|\nabla e^{y/6}|^6\big]
  \end{align}
  holds \emph{pointwise},
  i.e., $y$'s domain of definition plays no role here.
  The key idea in proving \eqref{eq:monster} is that after division by $e^y$,
  it becomes an inequality between polynomials in the first, second and third derivatives of $y$,
  hence the proof of its validity is a (cumbersome) algebraic problem.
  For more details on the choice of $\velo$
  and the general concepts behind the algebraic method for proving entropy dissipation inequalities,
  we refer the reader to \cite{BJM,JuMa}.

  Integration of \eqref{eq:monster} on $\setR^d$,
  recalling \eqref{eq:nrgdiss}, and making use of Lemma \ref{lem:Gauss},
  leads to \eqref{eq:cpriori}. 
\end{proof}
We are finally in the position to prove the main estimate \eqref{eq:bpriori}.
\begin{proof}[Proof of Proposition \ref{prp:apriori}]
  This is another application of the flow interchange method, see Lemma \ref{lem:flowinterchange}.
  According to Example \ref{xmp:flow},
  the heat flow given by $\flow_r\mu=K_r\ast\mu$ defines a $0$-flow on $\propo$ for the (unperturbed) entropy $\ent_0$.
  Therefore,
  \begin{align}
    \label{eq:fromevi}
    \liminf_{r\downarrow0}\left(-\frac1r\big[\nrgm_0(w_r\leb)-\nrgm_0(\mu_\tau^n)\big]\right)
    \le \frac{\ent_0(u_\tau^{n-1})-\ent_0(u_\tau^n)}\tau.
  \end{align}
  Formally, the left-hand side above is minus the derivative of $r\mapsto\nrgm_0(w_r\leb)$ at $r=0$,
  and the eventual goal is to control this in the spirit of \eqref{eq:cpriori}, for $\eps=0$.
  The main technical obstacle that prevents us from carrying out this differentiation 
  and to conclude directly \eqref{eq:bpriori} from here
  is the possible irregularity of $\nrgm_0(w_r\leb)$ at $r=0$.
  It is not even clear that $\nrgm_0(w_r\leb)\to\nrgm_0(\mu_\tau^n)$ as $r\downarrow0$.
  Indeed, while lower semi-continuity is known from Proposition \ref{prp:savarelsc},
  upper semi-continuity might fail.
  The problem is that --- due to the non-differentiability of $s\mapsto\sqrt[4]s$ at $s=0$ ---
  one cannot conclude from $\sqrt[4]{u_\tau^n}\in W^{1,4}(\setR^d)$
  that $\nabla\sqrt[4]{w_r}\to\nabla\sqrt[4]{u_\tau^n}$ in $L^4(\setR^d)$.
  We tackle this problem by approximation of $\nrgm_0$ by $\nrg^\eps$,
  for which continuity at $r=0$ has been shown in Lemma \ref{lem:epscont}.

  It follows from Lemma \ref{lem:cpriori} that $r\mapsto\nrg^\eps(w_r)$ is differentiable at every $r>0$,
  with derivative given in \eqref{eq:cpriori}.
  Thanks to Lemma \ref{lem:epscont}, we can apply the fundamental theorem of calculus to obtain that
  \begin{align*}
    -\frac1{\bar r}\big[\nrg^\eps(w_{\bar r})-\nrg^\eps(u_\tau^n)\big]
    & =\fint_0^{\bar r}\left(-\frac{\dn}{\dd r}\nrg^\eps(w_r)\right)\dd r \\
    &\ge \kappa\fint_0^{\bar r}\intrn\big(\interleave\nabla^3\sqrt{w_r+\eps}\interleave^2+|\nabla\sqrt[6]{w_r+\eps}|^6\big)\dd x\dd r
  \end{align*}
  We pass to the limit $\eps\downarrow0$, using Lemma \ref{lem:nrgeps} on the left hand side and Fatou's lemma respectively the lower semi-continuity of norms on the right hand side.
  With the latter we subsequently obtain for $\bar r\downarrow0$
  \begin{align*}
    \liminf_{\bar r\downarrow0}\left(-\frac1{\bar r}\big[\nrgm_0(w_{\bar r})-\nrgm_0(u_\tau^n)\big]\right)
    &\ge \kappa\liminf_{\bar r\downarrow0}\fint_0^{\bar r}\intrn\big(\interleave\nabla^3\sqrt{w_r}\interleave^2+|\nabla\sqrt[6]{w_r}|^6\big)\dd x\dd r \\
    &\ge \kappa\intrn\big(\interleave\nabla^3\sqrt{u_\tau^n}\interleave^2+|\nabla\sqrt[6]{u_\tau^n}|^6\big)\dd x.
  \end{align*}
  Plugging this into \eqref{eq:fromevi} yields \eqref{eq:bpriori} in the case $\lambda=0$.

  The changes induced by passing from $\lambda=0$ to $\lambda>0$ means
  to add on the right-hand side of \eqref{eq:cpriori} the contribution
  \begin{align*}
    \lambda^3\frac{\dn}{\dd r}\intrn |x|^2\dd\nu_r
    = \lambda^3\intrn|x|^2\Delta w_r \dd x
    = \lambda^3\intrn 2dw_r\dd x = 2d\lambda^3,
  \end{align*}
  which is independent of $r>0$.
  The passage to the limit $r\downarrow0$ is trivial here.
  In combination with the the result for $\lambda=0$,
  we arrive at \eqref{eq:bpriori}.
\end{proof}


\subsection{A universal bound on the energy}
In the spirit of \cite[Theorem 1.4]{GST}, we prove the following bound on the energy.
\begin{proposition}
  \label{prp:universal}
  There is a constant $\ebnd$, expressible in terms of
  the initial entropy value $\ent_0(u_0)$ and the initial second moment $\mom(u_0)$,
  such that, for each $N=1,2,\ldots$,
  \begin{align}
    \label{eq:universal}
    \nrgm_\lambda(\mu_\tau^N) \le \ebnd\big(1+(N\tau)^{-2/3}\big).
  \end{align}
\end{proposition}
This proposition is a consequence of the following a priori estimates.
\begin{lemma}
  \label{lem:impro}
  There is a universal constant $A$ such that,
  for each $N=1,2,\ldots$,
  \begin{align}
    \label{eq:improends}
    \frac\kappa2\tau \sum_{n=1}^N
    \intrn \big(\interleave\nabla^3 \sqrt{u_\tau^n}\interleave^2+|\nabla \sqrt[6]{u_\tau^n}|^6 \big)\dd x 
    &\le \pi\mom(u_0) + \ent_0(u_0) +(A+2d\lambda^3)N\tau, \\
    \label{eq:impromom}
    \pi\mom(u_\tau^N)
    &\le 2\pi\mom(u_0) + \ent_0(u_0) + 2(A+d\lambda^3)N\tau.
  \end{align}
\end{lemma}
As a technical ingredient in the proofs of both Proposition \ref{prp:universal} and Lemma \ref{lem:impro} above,
we need:
\begin{lemma}
  \label{lem:improtechnical}
  There exists a constant $B$ such that for any $\mu=u\leb\in\Dom{\nrgm_0}$,  
  \begin{align}
    \label{eq:GNS}
    \big[\nrgm_0(\mu)\big]^{3/2} \le B \intrn \interleave\nabla^3 \sqrt u\interleave^2\dd x.
  \end{align}
  And consequently, 
  for each $\eps>0$, there exists a $C_\eps$ independent of $\mu$ such that
  \begin{align}
    \label{eq:improtechnical}
    \nrgm_0(\mu) \le \eps\intrn \interleave\nabla^3 \sqrt u\interleave^2\dd x + C_\eps.
  \end{align}
\end{lemma}
\begin{proof}[Proof of Proposition \ref{prp:universal} from Lemmas \ref{lem:impro} and \ref{lem:improtechnical}]
  Without loss of generality, it suffices to prove
  \begin{align}
    \label{eq:pre-universal}
    \nrgm_\lambda(\mu) \le \ebnd\,(N\tau)^{-2/3}
  \end{align}
  for all $N$ such that $N\tau\le1$;
  the estimate \eqref{eq:universal} for larger $N$ is then a trivial consequence of the monotonicity \eqref{eq:ce-mono}.
  Choosing $N$ accordingly, the monotonicity \eqref{eq:ce-mono} of $\nrgm_\lambda$ implies that
  \begin{align*}
    N\tau\big[\nrgm_\lambda(\mu_\tau^N)\big]^{3/2}
    \le \tau\sum_{n=1}^N\big[\nrgm_\lambda(\mu_\tau^n)\big]^{3/2}.
  \end{align*}
  Substitute the elementary estimate
  \begin{align*}
    \big[\nrgm_\lambda(\mu_\tau^n)\big]^{3/2}
    \le \sqrt2\left(\nrgm_0(\mu_\tau^n)^{3/2} + \lambda^{9/2}\mom(\mu_\tau^n)^{3/2}\right)
  \end{align*}
  in the sum on the right-hand side and use \eqref{eq:GNS} to obtain
  \begin{align*}
    N\tau\big[\nrgm_\lambda(\mu_\tau^N)\big]^{3/2}
    \le \sqrt \tau\sum_{n=1}^N 2\left( B\intrn \interleave\nabla^3 \sqrt{u_\tau^n}\interleave^2\dd x + \lambda^{9/2}\mom(\mu_\tau^n)^{3/2}\right).
  \end{align*}
  Since we assume $n\tau\le N\tau\le 1$,
  the terms on the right-hand side are bounded thanks to \eqref{eq:improends} and \eqref{eq:impromom}.
  That is, $N\tau\big[\nrgm_\lambda(\mu_\tau^N)\big]^{3/2}\le C$,
  where $C$ depends just $\ent_0(u_0)$ and $\mom(u_0)$.
  Divide by $N\tau$ and take the power $2/3$
  to obtain \eqref{eq:pre-universal} with $\ebnd=C^{2/3}$.
\end{proof}
\begin{proof}[Proof of Lemma \ref{lem:impro}]
  Summation of \eqref{eq:bpriori} from $n=1$ to $n=N$ yields
  \begin{align}
    \label{eq:improstarts}
    \kappa\tau \sum_{n=1}^N
    \intrn \big(\interleave\nabla^3 \sqrt{u_\tau^n}\interleave^2
    +|\nabla \sqrt[6]{u_\tau^n}|^6 \big)\dd x
    \le \ent_0(u_0) - \ent_0(u_\tau^N) +2d\lambda^3N\tau.
  \end{align}
  We shall now derive a suitable lower bound on $\ent_0(u_\tau^N)$.
  More precisely, we derive an upper bound on $\mom(u_\tau^N)$;
  notice that entropy and second moment are connected via
  \begin{align}
    \label{eq:entbymom}
    - \ent_0(u_\tau^N) \le \pi\mom(u_\tau^N),
  \end{align}
  which follows from a scaling argument --- see e.g. \cite[Section 2.2]{GST} for details.
  To estimate the second moment, we apply once again the flow interchange technique.
  The flow $\flow_{(\cdot)}$ is given by exponential dilations, that is $\flow_\s\mu = (e^{-\s}\id)\#\mu$,
  or more explicitly in terms of the densities $u_\s$ of $\flow_\s\mu$:
  \begin{align*}
    u_\s (x) = e^{d\s} u(e^\s x).
  \end{align*}
  Since, by the chain rule,
  \begin{align*}
    \big[\nabla^2\log u_\s\big](x) = e^{2\s}\big[\nabla^2\log u\big](e^\s x),
  \end{align*}
  and since
  \begin{align*}
    \mom(u_\s) = \intrn|x|^2u_\s(x)\dd x = \intrn |e^{-\s}y|^2u(y)\dd y = e^{-2\s}\mom(u),
  \end{align*}
  it is easily seen that
  \[\nrg_\lambda(u_\s) = e^{4\s}\nrg_0(u) + \lambda^3e^{-2\s}\mom(u). \]
  Clearly, this scaling property carries over to the extension $\nrgm_\lambda$,
  and we obtain
  \begin{align*}
    \frac{\dn}{\dd\s}\bigg|_{\s=0}\nrgm_\lambda(\flow_\s\mu_\tau^n)
    = 4\nrgm_0(\mu_\tau^n) - 2\lambda^3\mom(\mu_\tau^n).
  \end{align*}
  Recall from Example \ref{xmp:flow} that $\flow_{(\cdot)}$ is a $1$-flow for the auxiliary functional $\frac12\mom$.
  The flow interchange estimate \eqref{eq:flowinterchange} now yields
  \begin{align*}
    -4\nrgm_0(\mu_\tau^n) + 2\lambda^3\mom(\mu_\tau^n)
    &= \limsup_{\s\to0}\frac{\nrgm_\lambda(\mu_\tau^n)-\nrgm_\lambda(\flow_\s\mu_\tau^n)}{\s} \\
    &\le \frac{\mom(\mu_\tau^{n-1})-\mom(\mu_\tau^n)}{2\tau} - \frac\tau2\left(\frac{\wassernoq{\mu_\tau^n}{\mu_\tau^{n-1}}}\tau\right)^2.
  \end{align*}
  Neglecting the last term, we obtain from here the recursion formula
  \begin{align*}
    \mom(\mu_\tau^n) \le \mom(\mu_\tau^{n-1}) + 8\tau\nrgm_0(\mu_\tau^n)-4\lambda^3\tau\mom(\mu_\tau^n),
  \end{align*}
  which clearly implies that
  \begin{align*}
    \mom(\mu_\tau^N) \le \mom(\mu_\tau^0) + 8\tau\sum_{n=1}^N\nrgm_0(\mu_\tau^n).
  \end{align*}
  On the right-hand side, we estimate $\nrgm_0(\mu_\tau^n)$
  by means of \eqref{eq:improtechnical} with $\eps:=\kappa/(16\pi)$,
  \begin{align}
    \label{eq:help001}
    \mom(\mu_\tau^N) \le \mom(\mu_\tau^0)
    + \frac{\kappa}{2\pi}\tau\sum_{n=1}^N\intrn\interleave \nabla^3\sqrt{u_\tau^n}\interleave^2\dd x + 8C_\eps N\tau.
  \end{align}
  By means of the inequality \eqref{eq:entbymom} between entropy and second moment,
  this provides an estimate from above on $-\ent_0(u_\tau^N)$ on the right-hand side of \eqref{eq:improstarts}.
  Rearranging terms, one finally obtains \eqref{eq:improends}, with $A:=8\pi C_\eps$.
  The estimate \eqref{eq:impromom} is now easily obtained
  by estimating the first sum in \eqref{eq:help001} above with \eqref{eq:improends},
  and neglecting the second sum.
\end{proof}
\begin{proof}[Proof of Lemma \ref{lem:improtechnical}]
  By standard interpolation of Sobolev norms, there exists a constant $A$ such that
  \begin{align*}
    \|\nabla^2f\|_{L^2}^2
    \le A\|\nabla^3f\|_{L^2}^{4/3}\|f\|_{L^2}^{2/3}
  \end{align*}
  for all $f\in W^{3,2}(\setR^d)$.
  Apply this to $f:=\sqrt u$.
  Since
  \begin{align*}
    \big\|\sqrt u\big\|_{L^2}^2 = \intrn u\dd x = 1,
  \end{align*}
  we obtain by means of \eqref{eq:nrgabove} that
  \begin{align*}
    \nrgm_0(\mu)\le C'A\big\|\nabla^3\sqrt u\big\|_{L^2}^{4/3},
  \end{align*}
  which is \eqref{eq:GNS}.
  The other estimate \eqref{eq:improtechnical} follows directly from \eqref{eq:GNS}
  via Young's inequality.
\end{proof}
\begin{corollary}
  There is a constant $C$, 
  expressible in terms of the initial entropy $\ent_0(u_0)$ and the initial second moment $\mom(u_0)$ alone,
  such that, for each $T>0$,
  \begin{align}
    \label{eq:zpriori}
    \int_0^T\intrn \big(\interleave\nabla^3 \sqrt{\bar u_\tau}\interleave^2
    +|\nabla \sqrt[6]{\bar u_\tau}|^6 \big)\dd x\dd t
    \le C(1+T)
  \end{align}
\end{corollary}
\begin{proof}
  This is a direct consequence of \eqref{eq:improends}.
\end{proof}


\subsection{Uniform almost continuity of the discrete trajectory}
In this section we show that the piecewise constant interpolations $\bar\mu_\tau$
are ``uniformly almost continuous'' as curves in $\propo$.
More specifically, we derive an approximate H\"older estimate with a $\tau$-independent H\"older constant,
see Proposition \ref{prp:holder} below.
\begin{lemma}
  \label{lem:1ststep}
  There is a constant $A$, depending on $\mu_0=u_0\leb$ just in terms of $\mom(\mu_0)$,
  such that for every $\tau\in(0,1)$:
  \begin{align}
    \label{eq:firststep}
    \wassernoq{\mu_\tau^1}{\mu_0} \le A\tau^{1/6}.
  \end{align}
\end{lemma}
\begin{proof}
  The idea is to show that, for a constant $A'$ expressible just in terms of $\mom(\mu_0)$,
  \begin{align}
    \label{eq:1ststep1}
    \frac1{2\tau}\wasser{\heat_\sigma\ast\mu_0}{\mu_0} + \nrg_\lambda(\heat_\sigma \ast u_0) \le A'\tau^{-2/3}.
  \end{align}
  It then follows thanks to non-negativity of $\nrgm_\lambda$
  that the minimizer $\mu_\tau^1$ of $\nrgm_{\lambda,\tau} (\cdot;\mu_0)$ satisfies
  \begin{align*}
    \frac1{2\tau}\wasser{\mu_\tau^1}{\mu_0}
    \le \nrgm_{\lambda,\tau}(\mu_\tau^1;\mu_0)
    \le \nrgm_{\lambda,\tau}(\heat_\sigma\ast\mu_\tau^0;\mu_0)
    \le A'\tau^{-2/3},
  \end{align*}
  which immediately implies \eqref{eq:firststep}.
  To prove \eqref{eq:1ststep1}, we show that, on the one hand,
  \begin{align}
    \label{eq:1ststep2}
    \wasser{\heat_\sigma\ast\mu_0}{\mu_0} \le 2d\sigma,
  \end{align}
  and that, on the other hand,
  \begin{align}
    \label{eq:1ststep3}
    \nrg_\lambda(\heat_\sigma\ast u_0) \le B\sigma^{-2},
  \end{align}
  for all $\sigma\in(0,1)$, with a constant $B$ that again depends on $\mu_0$ only via $\mom(\mu_0)$.
  The choice $\sigma:=\tau^{1/3}$ then yields \eqref{eq:1ststep1}.
  For the proof of \eqref{eq:1ststep2}, we compare the Wasserstein distance
  with the transport cost generated by the plan $\pi$ that has Lebesgue density
  \begin{align*}
    g(x,y) = u_0(x)\heat_\sigma(y-x).
  \end{align*}
  It is easily seen that the two marginals are indeed $u_0$ and $\heat_\sigma\ast u_0$, respectively.
  According to \eqref{eq:Wassersteindist1}, we have
  \begin{align*}
    \wasser{\heat_\sigma\ast\mu_0}{\mu_0}
    \le \int_{\setR^d\times\setR^d}|x-y|^2 g(x,y)\dd(x,y)
    = \intrn u_0(x)\dd x \intrn |z|^2\heat_\sigma(z)\dd z
    = 2d\sigma.
  \end{align*}
  The proof of \eqref{eq:1ststep3} is a bit more elaborate.
  First, note that
  \begin{align*}
    \heat_\sigma\ast\big(|\cdot|^2\big)(y)
    &= \intrn |y-x|^2\heat_\sigma(x)\dd x \\
    & = |y|^2\intrn\heat_\sigma(x)\dd x - 2y\cdot\intrn x\heat_\sigma(x)\dd x + \intrn|x|^2\heat_\sigma(x)\dd x \\
    & = |y|^2 - 2y\cdot 0 + 2d\sigma,
  \end{align*}
  and hence
  \begin{align*}
    \mom(\heat_\sigma\ast u_0)
    = \intrn |x|^2 \dd\big(\heat_\sigma\ast u_0\big)(x)
    &= \intrn \heat_\sigma\ast\big(|\cdot|^2\big)(y)\dd \mu_0(y) \\
    &= 2d\sigma \intrn\dd\mu_0(y) + \intrn|y|^2\dd\mu_0(y)
    = 2d\sigma + \mom(\mu_0).
  \end{align*}
  Concerning the estimate of $\nrg_0$, observe that for a smooth and positive probability density $f$,
  \begin{align*}
    f\big\|\nabla^2\log f\big\|^2
    = f\left\|\frac{\nabla^2f}{f} - \frac{\nabla f\otimes\nabla f}{f^2}\right\|^2
    \le 2\frac{\|\nabla^2f\|^2}{f} + 2\frac{|\nabla f|^4}{f^3}.
  \end{align*}
  Now we plug $f:=\heat_\sigma\ast u_0$ in.
  By Jensen's inequality,
  \begin{align*}
    \big\|\nabla^2(\heat_\sigma\ast u_0)\big\|^2(x)
    &= \big(\heat_\sigma\ast u_0\big)^2(x)
    \left\|\intrn\frac{\nabla^2\heat_\sigma(y)}{\heat_\sigma(y)}\,\frac{\heat_\sigma(y)u_0(x-y)\dd y}{\big(\heat_\sigma\ast u_0\big)(x)}\right\|^2 \\
    & \le \big(\heat_\sigma\ast u_0\big)^2(x)
    \intrn \left\|\frac{\nabla^2\heat_\sigma(y)}{\heat_\sigma(y)}\right\|^2 \,\frac{\heat_\sigma(y)u_0(x-y)\dd y}{\big(\heat_\sigma\ast u_0\big)(x)} \\
    &= \big(\heat_\sigma\ast u_0\big)(x)\,
      \left(\frac{\|\nabla^2\heat_\sigma\|^2}{\heat_\sigma}\ast u_0\right) (x).
  \end{align*}
  It thus follows that   
  \begin{align*}
    \intrn \frac{\|\nabla^2(\heat_\sigma\ast u_0)\|^2}{\heat_\sigma\ast u_0}\dd x
    \le \intrn \left(\frac{\|\nabla^2\heat_\sigma\|^2}{\heat_\sigma}\ast u_0\right)\dd x
    = \intrn \frac{\|\nabla^2\heat_\sigma\|^2}{\heat_\sigma} \dd x
    = \left(\frac1{2\sigma}\right)^2d(d+1).
  \end{align*}
  In an analogous manner,
  \begin{align*}
    \big|\nabla(\heat_\sigma\ast u_0)\big|^4(x)
    &=  \big(\heat_\sigma\ast u_0\big)^4(x)\,
      \left|\intrn \frac{\nabla\heat_\sigma(y)}{\heat_\sigma(y)}\,\frac{\heat_\sigma(y)u_0(x-y)\dd y}{\big(\heat_\sigma\ast u_0\big)(x)}\right|^4 \\
    &\le \big(\heat_\sigma\ast u_0\big)^4(x)\,
      \intrn\left|\frac{\nabla\heat_\sigma(y)}{\heat_\sigma(y)}\right|^4 \,\frac{\heat_\sigma(y)u_0(x-y)\dd y}{\big(\heat_\sigma\ast u_0\big)(x)} \\
    &= \big(\heat_\sigma\ast u_0\big)^3(x)\,
      \left(\frac{|\nabla\heat_\sigma|^4}{\heat_\sigma^3}\ast u_0\right)(x),
  \end{align*}
  implying that   
  \begin{align*}
    \intrn \frac{|\nabla(\heat_\sigma\ast u_0)|^4}{(\heat_\sigma\ast u_0)^3}\dd x
    \le \intrn \left(\frac{|\nabla\heat_\sigma|^4}{\heat_\sigma^3}\ast u_0\right)\dd x
    = \intrn \frac{|\nabla\heat_\sigma|^4}{\heat_\sigma^3} \dd x
    = \left(\frac1{2\sigma}\right)^2d(d+2).
  \end{align*}
  Collecting these estimates, we finally obtain
  \begin{align*}
    \nrg_\lambda(\heat_\sigma\ast u_0)
    & = \nrg_0(\heat_\sigma\ast u_0) + \lambda^3\mom(\heat_\sigma\ast u_0) \\
    & \le \intrn\frac{\|\nabla^2(\heat_\sigma\ast u_0)\|^2}{\heat_\sigma\ast u_0}\dd x
      + \intrn\frac{|\nabla(\heat_\sigma\ast u_0)|^4}{(\heat_\sigma\ast u_0)^3}\dd x
      +\lambda^3\mom(\heat_\sigma\ast u_0) \\
    & \le \left(\frac1{2\sigma}\right)^2d(2d+3) + \lambda^3\big[2d\sigma + \mom(u_0)\big].
  \end{align*}
  From here, \eqref{eq:1ststep3} follows immediately.
\end{proof}
\begin{lemma}
  \label{lem:holder1}
  There is a constant $\lip$ such that, for all $\overline M>\underline M\ge 1$ with $\overline M\tau\le1$,
  \begin{align}
    \label{eq:holder1}
    \wassernoq{u_\tau^{\overline M}}{u_\tau^{\underline M}} \le \lip \big(\overline M\tau-\underline M\tau\big)^{1/12}.
  \end{align}
\end{lemma}
\begin{proof}
  Let $N_\tau$ be the largest integer $N$ with $N\tau\le1$.
  Using Lemma \ref{lem:elementary1} from the appendix and Proposition \ref{prp:universal},
  we find that
  \begin{align*}
    \tau\sum_{n=1}^{N_\tau} (n\tau)^{5/6}\left(\frac{\wassernoq{u_\tau^{n+1}}{u_\tau^n}}\tau\right)^2
    & \le \tau\sum_{n=1}^{N_\tau}\left[\tau\sum_{N=1}^n(\tau N)^{-1/6} \left(\frac{\wassernoq{u_\tau^{n+1}}{u_\tau^n}}\tau\right)^2\right] \\
    &\le \sum_{N=1}^{N_\tau}\left[(\tau N)^{-1/6}\ \tau\sum_{n=N}^\infty \left(\frac{\wassernoq{u_\tau^{n+1}}{u_\tau^n}}\tau\right)^2\right] \\
    &\le 2\sum_{N=1}^{N_\tau} (\tau N)^{-1/6}\nrgm_\lambda(u_\tau^N) \\
    &\le 4\ebnd\sum_{N=1}^{N_\tau} (\tau N)^{-5/6}
    \le 4\ebnd \int_0^1s^{-5/6}\dd s = 24\ebnd.
  \end{align*}
  Now we combine this with the triangle inequality for $\wass$,
  the basic energy estimate \eqref{eq:ce-sum}, and
  Lemma \ref{lem:elementary2} from the appendix:
  \begin{align*}
    \wassernoq{u_\tau^{\overline M}}{u_\tau^{\underline N}}
    &\le \sum_{n=\underline M}^{\overline M-1}\wassernoq{u_\tau^{n+1}}{u_\tau^n} \\
    &\le \left(\tau \sum_{n=1}^{N_\tau}(n\tau)^{5/6}\left(\frac{\wassernoq{u_\tau^{n+1}}{u_\tau^n}}\tau\right)^2 \right)^{1/2} 
    \left(\tau\sum_{n=\underline M}^{\overline M-1}(n\tau)^{-5/6}\right)^{1/2} \\
    &\le (24\ebnd)^{1/2} \left( 12\left[\big(\overline M \tau\big)^{1/6}-\big(\underline M\tau\big)^{1/6}\right] \right)^{1/2}.
  \end{align*}
  To conclude \eqref{eq:holder1} from here, with $\lip=12\sqrt{2\ebnd}$,
  it suffices to recall that $(a+b)^{1/6}\le a^{1/6}+b^{1/6}$ for arbitrary non-negative reals $a$ and $b$. 
\end{proof}
\begin{lemma}
  \label{lem:holder2}
  For all $\overline M\ge\underline M$ with $\underline M\tau\ge1$,
  \begin{align}
    \label{eq:holder2}
    \wassernoq{\mu_\tau^{\overline M}}{\mu_\tau^{\underline M}} \le 2\sqrt{\ebnd}\big(\overline M\tau-\underline M\tau\big)^{1/2}.
  \end{align}
\end{lemma}
\begin{proof}
  From the basic estimate \eqref{eq:ce-sum},
  we obtain that
  \begin{align*}
     \wassernoq{u_\tau^{\overline M}}{u_\tau^{\underline N}}
    &\le \sum_{n=\underline M}^{\overline M-1}\wassernoq{u_\tau^{n+1}}{u_\tau^n} \\
    &\le \left(\tau \sum_{n=\underline M}^{N_\tau}\left(\frac{\wassernoq{u_\tau^{n+1}}{u_\tau^n}}\tau\right)^2 \right)^{1/2} 
      \left(\tau\sum_{n=\underline M}^{\overline M-1}1\right)^{1/2} \\
    &\le \big(2\nrgm_\lambda(\mu_\tau^{\underline M})\big)^{1/2}\big(\overline M\tau-\underline M\tau\big)^{1/2}
  \end{align*}
  To conclude \eqref{eq:holder2}, observe that $\nrgm_\lambda(\mu_\tau^{\underline M}) \le 2\ebnd$.
  thanks to Proposition \eqref{prp:universal} and since $\underline M\tau\ge1$.
\end{proof}
Recall that a \emph{modulus of continuity} is a map $\omega:\setR_{\ge0}\times\setR_{\ge0}\to\setR_{\ge0}$
with the property that $\lim_{(s,t)\to(r,r)}\omega(s,t)=0$ for arbitrary $r\in\setR_{\ge0}$.
\begin{proposition}
  \label{prp:holder}
  With the constants $A$ from Lemma \ref{lem:1ststep},
  $\lip$ from Lemma \ref{lem:holder1},
  and $\ebnd$ from Proposition \ref{prp:universal},
  one has, for all $\tau\in(0,1)$,
  \begin{align}
    \label{eq:holderall}
    \wassernoq{\bar\mu_\tau(t)}{\bar\mu_\tau(s)}
    \le \left(A+\lip+2\sqrt{\ebnd}\right)\big(\tau^{1/12} + |t-s|^{1/12} + |t-s|^{1/2}\big) \quad \text{for all $s,t\ge0$}.
  \end{align}
\end{proposition}
\begin{proof}
  Without loss of generality, assume that $t>s$.

  Given $\tau$, let $\underline M$ and $\overline M$ be the smallest integers
  with $\underline M\tau\ge s$ and $\overline M\tau\ge t$, respectively.
  By definition of $\bar\mu_\tau$, we have that
  \begin{align*}
    \wassernoq{\bar\mu_\tau(t)}{\bar\mu_\tau(s)}
    = \wassernoq{\mu_\tau^{\overline M}}{\mu_\tau^{\underline M}}.
  \end{align*}
  Further note that $\overline M\tau-\underline M\tau \le t-s+\tau$.  
  Let $N_\tau$ be the smallest integer $N$ with $N\tau\ge1$.
  If either $\underline M\ge N_\tau$ or $\overline M\le N_\tau$,
  then \eqref{eq:holderall} follows directly from \eqref{eq:holder1} or \eqref{eq:holder2}, respectively,
  using that
  \begin{align*}
    \big(\overline M\tau-\underline M\tau\big)^{1/12} &\le (t-s+\tau)^{1/12} \le (t-s)^{1/12} + \tau^{1/12} , \\
    \big(\overline M\tau-\underline M\tau\big)^{1/2} &\le (t-s+\tau)^{1/2} \le (t-s)^{1/2} + \tau^{1/2} \le (t-s)^{1/2} + \tau^{1/12}.
  \end{align*}
  If instead $\underline M<N_\tau<\overline M$, then we estimate further:
  \begin{align*}
    \wassernoq{\mu_\tau^{\overline M}}{\mu_\tau^{\underline M}}
    \le \wassernoq{\mu_\tau^{\overline M}}{\mu_\tau^{N_\tau}} + \wassernoq{\mu_\tau^{N_\tau}}{\mu_\tau^{\underline N}},
  \end{align*}
  and apply \eqref{eq:holder1} and \eqref{eq:holder2}, respectively, to the sum on the right-hand side,
  using that, trivially,
  \begin{align*}
    \overline M\tau-N_\tau \le t-s+\tau,\quad
    N_\tau-\underline M\tau \le t-s.
  \end{align*}
  Finally, if $\underline M_\tau$, i.e., $s=0$,
  then we estimate
  \begin{align*}
    \wassernoq{\mu_\tau^{\overline M}}{\mu_0}
    \le \wassernoq{\mu_\tau^{\overline M}}{\mu_\tau^1} + \wassernoq{\mu_\tau^1}{\mu_0},
  \end{align*}
  apply the reasoning above to the first distance, and Lemma \ref{lem:1ststep} to the second,
  using $\tau^{1/6}\le\tau^{1/12}$.
\end{proof}

        
\subsection{Passage to the continuous equation}
The collected estimates will be enough to pass to the limit $\tau \downarrow 0$
in the time-discrete evolution equation \eqref{eq:discweak}.
We choose for any $\tau>0$ an interpolated time discrete solution $\bar\mu_\tau:\setR_{\ge0}\to\propo$
starting from $\mu_\tau^0=u_0\leb$,
with respective Lebesgue densities $\bar u_\tau:\setR_{\ge0}\to L^1(\rn)$.
\begin{lemma}
  \label{lem:limit}
  There is a sequence $\tau_k\downarrow0$ and a H\"older continuous limit curve $u\leb:\setR_{\ge0}\to (\propo, \wass)$
  with $u(0,\cdot)=u_0$ and $u \in L_{loc}^2(\mathbb{R}_{>0}; W^{2,2}(\rn))$,
  such that 
  \begin{enumerate}   
  \item[i)] \label{it:conv-n} $\bar u_{\tau_k}(t, \cdot)\leb \rightarrow u(t, \cdot) \leb$ narrowly at each $t \in [0,T]$,
  \item[ii)] \label{it:conv-s} $\sqrt{\bar u_{\tau_k}} \rightarrow \sqrt{u}$ strongly in $L_{loc}^2(\mathbb{R}_{>0}; W^{2,2}(\rn))$,
  \item[iii)] \label{it:conv-z} $\sqrt[4]{\bar u_{\tau_k}} \rightarrow \sqrt[4]{u}$ strongly in $L_{loc}^4(\mathbb{R}_{>0}; W^{1,4}(\rn))$, 
  \item[iv)] \label{it:conv-6} $\sqrt[6]{\bar u_{\tau_k}} \rightarrow \sqrt[6]{u}$ weakly in $L_{loc}^6(\mathbb{R}_{>0}; W^{1,6}(\rn))$.
  \end{enumerate}
\end{lemma}
\begin{proof}
  In the following, let some time horizon $T>0$ be fixed.
  Thanks to the moment estimate \eqref{eq:impromom} and the $\tau$-uniform H\"older regularity \eqref{eq:holderall},
  the curves $\bar\mu_\tau$ satisfy the hypotheses of the generalized Arzel\'{a}-Ascoli-Theorem \cite[Theorem 3.3.1]{AGS}, see Lemma \ref{lem:arzela-ascoli}.
  
  Hence, for a suitable vanishing sequence $\tau_k$, the $\bar\mu_{\tau_k}$ converge
  --- narrowly at each $t\ge0$ --- to a limit curve $\mu:\setR_{\ge0}\to\propo$.
  And that limit inherits the H\"older continuity \eqref{eq:holderall}, i.e.,
  \begin{align}
    \label{eq:holder3}
    \wass(\mu(t), \mu(s)) \le C\big(|t-s|^{1/12}+|t-s|^{1/2}\big) \quad \text{for all $s,t\ge0$}.
  \end{align}
  By the lower semi-continuity of $\nrgm_\lambda$, see Proposition \ref{prp:savarelsc},
  and the universal bound \eqref{eq:universal},
  one has
  \begin{align*}
    \nrgm_\lambda(\mu(t)) \le \liminf_{k\to\infty}\nrgm_\lambda(\bar\mu_{\tau_k}(t)) \le \ebnd\big(1+t^{-2/3}\big)
  \end{align*}
  at each $t>0$.
  Hence the limit measures are absolutely continuous, $\mu(t)=u(t)\leb$, with $\sqrt{u(t)}\in W^{2,2}(\rn)$.
  Additionally, in view of the estimate \eqref{eq:savarebnd}, we have that
  \begin{align}
    \label{eq:conv-W22}
    \|\sqrt{\bar u_\tau (t)}\|_{W^{2,2}}^2 \le C\ebnd\big(1+t^{-2/3}\big),
  \end{align}
  and consequently,
  \begin{align}
    \label{eq:conv-W12p}
    \sqrt{\bar u_{\tau_k}(t)}\to\sqrt{u(t)}\quad \text{in $W^{1,2}(\rn)$}
  \end{align}
  at each $t>0$.    
  Indeed, by Rellich's theorem, a subsubsequence of an arbitrary subsequence of $\sqrt{\bar u_{\tau_k}(t)}$
  converges strongly to \emph{some} limit $v$ in $W^{1,2}(\rn)$;
  but then the implied pointwise a.e.\ convergence leads to $v^2=u(t)$.
  By independence of the limit from the chosen subsequence, we conclude convergence of the entire sequence.
  Next, since \eqref{eq:conv-W22} provides a uniform bound on $\sqrt{\bar u_\tau}$ in $L^2(0,T;W^{1,2}(\rn))$,
  the dominated convergence theorem applies and yields
  \begin{align*}
    \sqrt{\bar u_{\tau_k}}\to \sqrt{u}\quad\text{in $L^2(0,T;W^{1,2}(\rn)$}.
  \end{align*}
  We combine this convergence with the uniform bound on $\sqrt{\bar u_\tau}$ in $L^2(0,T;W^{3,2}(\rn))$ from \eqref{eq:zpriori}
  to obtain claim ii) above via interpolation.
  With ii) and the uniform bound on $\sqrt[6]{\bar u_\tau}$ in $L^6(0,T;W^{1,6}(\rn))$ from \eqref{eq:zpriori} at hand,
  we verify claim iii) via Theorem \ref{thm:interpol} from the Appendix.
  The last claim iv) is another direct consequence of the bound \eqref{eq:zpriori} on $\sqrt[6]{\bar u_\tau}$. That $u \in L_{loc}^2(\mathbb{R}_{>0}; W^{2,2}(\rn))$ follows from the arguments given in Lemma \ref{lem:epscont}.

\end{proof}
\begin{lemma}
  \label{lem:limit2}
  The limit $u$ defined in Lemma \ref{lem:limit} above is a solution to \eqref{eq:QDD6intro} in the sense of \eqref{eq:weak}.            
\end{lemma}
\begin{proof}
  Let $\psi\in C^\infty_c(\setR_{>0}\times\setR^d)$ be a test function in time and space.
  Fix some $\tau_k$;
  without loss of generality, we assume that $\tau_k$ is so small that $\psi(t,x)=0$ for all $0<t<\tau_k$ and $x\in\setR^d$.
  For each $n=1,2,\ldots$, use $\varphi_{\tau_k}^n:=\psi(n\tau_k;\cdot)$ as test function in \eqref{eq:discweak},
  then sum over all $n\in\setN$;
  this is actually a \emph{finite} sum since $\psi$ is compactly supported.
  With the help of the triangle inequality,
  \begin{align*}
    \left|-\tau\sum_{n=1}^\infty\intrn\frac{\varphi_{\tau_k}^{n+1}-\varphi_{\tau_k}^n}{\tau_k}u_{\tau_k}^n\dd x
    + \tau\sum_{n=1}^\infty\left(\N{u_{\tau_k}^n}{\varphi_{\tau_k}^n}+2\lambda^3\intrn x\cdot\nabla\varphi_{\tau_k}^nu_{\tau_k}^n\dd x\right)\right| \\
    \le \frac{\alpha\tau_k}2\sum_{n=1}^\infty\left(\frac{\wassernoq{u_{\tau_k}^n}{u_{\tau_k}^{n-1}}}{\tau_k}\right)^2.
  \end{align*}
  The right-hand side converges to zero for $k\to\infty$ thanks to the estimate \eqref{eq:ce-sum}.
  This implies, after rewriting everything in terms of the interpolated functions,
  \begin{align*}
    \lim_{k\to\infty}\int_0^T\intrn \delta_{\tau_k}\psi\,\bar u_{\tau_k}\dd x\dd t
    = \lim_{k\to\infty}\int_0^T\left(\N{\bar u_{\tau_k}}{\bar\psi_{\tau_k}}+2\lambda^3\intrn x\cdot\nabla\bar\psi_{\tau_k}\bar u_{\tau_k}\dd x\right)\dd t,
  \end{align*}
  where $T>0$ is chosen large enough so that $\operatorname{supp}\psi\subset(0,T)\times\setR^d$,
  and we have introduced
  \begin{align*}
    \bar\psi_{\tau_k}(t) = \psi(n\tau_k), \quad
    \delta_{\tau_k}\psi(t) = \frac{\psi((n+1)\tau_k)-\psi(n\tau_k)}{\tau_k} \quad
    \text{for all $t\in((n-1]\tau_k,n\tau_k]$}.
  \end{align*}
  Notice that
  \begin{align*}
    \bar\psi_{\tau_k}\to\psi, \quad
    \delta_{\tau_k}\psi \to\partial_t\psi \quad
    \text{uniformly on $\setR_{>0}\times\setR^d$}.
  \end{align*}
  It is now easily checked that the convergence stated in Lemma \ref{lem:limit} above are sufficient
  to pass to the respective limits inside the integrals,
  that is
  \begin{align*}
    \int_0^T\intrn\partial_t\psi\,u\dd x\dd t
    = \int_0^T\left(\N{u}{\psi}+2\lambda^3\intrn x\cdot\nabla\psi\,u\dd x\right)\dd t.
  \end{align*}
  This is equivalent to the weak formulation \eqref{eq:weak}.
\end{proof}
This finishes the proof of Theorem \ref{thm:existence}.

\section{Long Time Behaviour}
\label{sct:longtime}

        \subsection{An illustration by ODEs}
        \label{sct:illustration}
We illustrate the role played by \eqref{eq:justmagic}
by an analogous situation for smooth gradient flows on $\setR^n$.
We are given a family of strictly convex functions $h_\lambda:\setR^n\to\setR$ and fix a parameter $\lambda>0$.
From here, we define derived functions $f_\lambda,e_\lambda:\setR^n\to\setR$ by
\begin{align}
  \label{eq:justmagicODE}
  f_\lambda:=\frac12|\dff h_\lambda|^2,
  \quad
  e_\lambda:=\frac12\dff f_\lambda\cdot\dff h_\lambda + \lambda f_\lambda
  = \frac12\dff h_\lambda\cdot\dff^2h_\lambda\cdot\dff h_\lambda + \frac\lambda2|\dff h_\lambda|^2.
\end{align}
Thanks to strict convexity of $h_\lambda$,
there exists precisely one minimum point $x_\lambda$ of $h_\lambda$,
and this is by construction also the unique minimum point of $f_\lambda$ and of $e_\lambda$.
We wish to study the linearized dynamics of the gradient flow $\dot x=-\dff e_\lambda(x)$ near the stationary point $x_\lambda$.
Since
\begin{align*}
  \dff e_\lambda = \frac12\dff h_\lambda\cdot\dff^3h_\lambda\cdot\dff h_\lambda
  + \dff^2h_\lambda\cdot\dff^2h_\lambda\cdot\dff h_\lambda + \lambda\dff^2h_\lambda\cdot\dff h_\lambda,
\end{align*}
and since $\dff h_\lambda(x_\lambda)=0$,
it follows that the linearization $A_\lambda\in\setR^{n\times n}$
of $e_\lambda$'s gradient vector field $\dff e_\lambda(x)$ near $x=x_\lambda$ is given by
\begin{align*}
  A_\lambda\xi = \big[\dff^2h_\lambda(x_\lambda)\big]^3\xi + \lambda\big[\dff^2h_\lambda(x_\lambda)\big]^2\xi.
\end{align*}
Assuming that the eigenvalues $\mu_1,\ldots,\mu_n$ of $\dff^2h_\lambda(x_\lambda)$ are known,
the eigenvalue of $A$ are known as well:
these are precisely $\mu_k^3+\lambda\mu_k^2$ for $k=1,2,\ldots,n$.
In particular, the smallest of these is a lower bound on the exponential rate of convergence to $x_\lambda$
in the linearized dynamics.

 
\subsection{Derivation of the relation \eqref{eq:justmagic}}
\label{sct:HFE}
We shall now derive the relation \eqref{eq:justmagic},
which is the basis for all further analysis below,
and which plays the same role for \eqref{eq:QDD6intro}
as \eqref{eq:justmagicODE} has played in the analysis of the toy problem above.
\begin{lemma}
  Let $w_r$ be a solution to the linear Fokker-Planck equation \eqref{eq:LFP}.
  Then, at each $r>0$,
  \begin{align}
    \label{eq:magicH0}
    \frac12\frac{\dn}{\dd r}\ent_\lambda(w_r)
    &= -\intrn w_r\left|\nabla\log\frac{w_r}{U_\lambda}\right|^2\dd x, \\
    \nonumber
    \frac14\frac{\dn^2}{\dd r^2}\ent_\lambda(w_r)
    &=-\frac14\intrn \dv\left(w_r\nabla\log\frac{w_r}{U_\lambda}\right) \left|\nabla\log\frac{w_r}{U_\lambda}\right|^2\dd x \\
    \label{eq:magicF0}
    &\qquad + \frac12\intrn \frac1{w_r}\left[\dv\left(w_r\nabla\log\frac{w_r}{U_\lambda}\right)\right]^2\dd x.
  \end{align}                  
\end{lemma}
\begin{proof}
  Thanks to the regularizing properties of the linear Fokker-Planck equation,
  $(r,x)\mapsto w_r(x)$ is an everywhere positive $C^\infty$-function on $\setR_{>0}\times\setR^d$,
  where it satisfies \eqref{eq:LFP} in the classical sense,
  or equivalently
  \begin{align}
    \label{eq:LFP2}
    \partial_rw_r 
    = \dv\left(w_r\nabla\log\frac{w_r}{U_\lambda}\right).
  \end{align}
  Moreover, $w_r$ is a probability density at any $r>0$,
  and $w_r(x)$ decays sufficiently rapidly for $|x|\to\infty$ to justify the integration by parts below.
  
  To derive $\ent_\lambda$, we rewrite it in the form
  \begin{align}
    \label{eq:relativeH}
    \ent_\lambda(w) = \intrn \frac{w}{U_\lambda}\log\frac{w}{U_\lambda}\,U_\lambda\dd x.
  \end{align}
  For its $r$-derivative, we obtain after subsitution of \eqref{eq:LFP2} and an integration by parts:
  \begin{align*}
    \frac12\frac{\dn}{\dd r}\ent_\lambda(w_r)
    = \frac12 \intrn\left(1+\log\frac{w_r}{U_\lambda}\right)\,\partial_rw_r\dd x
    = -\frac12 \intrn w_r \left|\nabla\log\frac{w_r}{U_\lambda}\right|^2\dd x,
  \end{align*}
  which is \eqref{eq:magicH0}.
  
  For computation of the second $r$-derivative of $\ent_\lambda$,
  we differentiate in \eqref{eq:magicH0}, substitute \eqref{eq:LFP2} again, and integrate by parts
  to obtain
  \begin{align*}
    \frac14\frac{\dn^2}{\dd r^2}\ent_\lambda(w_r)
    &= -\frac14\frac{\dn}{\dd r}\intrn w_r \left|\nabla\log\frac{w_r}{U_\lambda}\right|^2\dd x \\
    &= -\frac14\intrn \partial_rw_r\, \left|\nabla\log\frac{w_r}{U_\lambda}\right|^2\dd x
      - \frac12\intrn w_r \,\nabla\left[\log\frac{w_r}{U_\lambda}\right]\cdot\nabla\left[\frac{\partial_rw_r}{w_r}\right]\dd x \\
    &= -\frac14\intrn \dv\left(w_r\nabla\log\frac{w_r}{U_\lambda}\right) \left|\nabla\log\frac{w_r}{U_\lambda}\right|^2\dd x
      + \frac12\intrn \frac1{w_r} \,\left[\dv\left(w_r\nabla\log\frac{w_r}{U_\lambda}\right) \right]^2\dd x,
  \end{align*}
  and this is \eqref{eq:magicF0}.
\end{proof}
\begin{lemma}
  Let $w_r$ be a solution to the linear Fokker-Planck equation \eqref{eq:LFP}.
  Then
  \begin{align}
    \label{eq:magicH}
    -\frac12\frac{\dn}{\dd r}\ent_\lambda(w_r) &= \fish_\lambda(w_r) - \fish_\lambda(U_\lambda), \\
    \label{eq:magicF}
    -\frac12\frac{\dn}{\dd r}\fish_\lambda(w_r) &= \nrg_\lambda(w_r) - \nrg(U_\lambda) -  \lambda\big[\fish_\lambda(w_r) - \fish(U_\lambda)\big],
  \end{align}
  for all $r>0$.
\end{lemma}
\begin{proof}
  As in the previous proof, we rely on the regularity of the Fokker-Planck flow in the calculations below.
  Still, to enhance readability, we shall use simply $u$ instead of $w_r$ below.
  
  To establish the connection of the right hand side in \eqref{eq:magicH0} to $\fish_\lambda$,
  we substitute
  \begin{align}
    \label{eq:explicit}
    \nabla\log\frac{u}{U_\lambda} = \nabla \log u+\lambda x,
  \end{align}
  and then integrate by parts in the term with linear dependence on $x$:
  \begin{align*}
    \frac12\intrn u|[\nabla\log u+\lambda x]|^2\dd x
    &= \frac12\intrn u|\nabla\log u|^2\dd x + \frac{\lambda^2}2\intrn |x|^2 u\dd x+ \lambda \intrn x\cdot\nabla u\dd x\\
    &= \fish_\lambda(u) - d\lambda\intrn u\dd x.
  \end{align*} 
  Since $u$ is a probability density, the constant above amounts to $d\lambda$.
  To conclude \eqref{eq:magicH} from here,
  it remains to observe that
  \begin{align*}
    \fish_\lambda(U_\lambda)
    = \frac12\intrn U_\lambda|\nabla\log U_\lambda|^2 \dd x + \frac{\lambda^2}2\intrn |x|^2U_\lambda\dd x
    = \lambda^2\intrn |x|^2U_\lambda\dd x = d\lambda.
  \end{align*}
  Next we need to show that the right-hand sides in \eqref{eq:magicF0} and \eqref{eq:magicF}, respectively, are the same.
  Substitution of \eqref{eq:explicit} into \eqref{eq:magicF0} yields, after an integration by parts,
  \begin{align*}
    -\frac12\frac{\dn}{\dd r}\fish_\lambda(u)
    &=-\frac14\intrn\big|[\nabla\log u+\lambda x]\big|^2\dv(u[\nabla \log u+\lambda x])\dd x \\
    &\qquad -\frac12\intrn u[\nabla\log u+\lambda x]\cdot\nabla\left(\frac{\dv(u[\nabla\log u+\lambda x])}u\right)\dd x \\
    &=-\frac14\intrn\big|[\nabla\log u+\lambda x]\big|^2\dv(u[\nabla \log u+\lambda x])\dd x \\
    &\qquad +\frac12\intrn [\nabla\log u+\lambda x]\cdot\nabla\log u\, \dv(u[\nabla\log u+\lambda x])\dd x \\
    &\qquad - \frac12\intrn[\nabla\log u+\lambda x]\cdot\nabla \dv(u[\nabla\log u+\lambda x])\dd x \\
    &=\frac14\intrn (|\nabla\log u|^2-\lambda^2|x|^2)\,\dv(u[\nabla\log u+\lambda x])\dd x \\
    &\qquad - \frac12\intrn[\nabla\log u+\lambda x]\cdot\dv\big\{\nabla\otimes(u[\nabla\log u+\lambda x])\big\}\dd x.
  \end{align*}
  Now we integrate by parts to remove the divergence in both integrals:
  \begin{align*}
    -\frac12\frac{\dn}{\dd r}\fish_\lambda(u)
    &=-\frac12\intrn u\nabla\log u\cdot\nabla^2\log u\cdot[\nabla\log u+\lambda x]\dd x
      +\frac{\lambda^2}2\intrn x\cdot\nabla u \dd x + \frac{\lambda^3}2\intrn u|x|^2\dd x \\
    &\qquad + \frac12\intrn (\nabla^2\log u+\lambda\eins):\big(\nabla u\otimes[\nabla\log u+\lambda x]\big)\dd x
      +\frac12 \intrn u\|\nabla^2\log u+\lambda\eins\|^2\dd x \\
    &= -\frac{d\lambda^2}2\intrn u\dd x + \frac{\lambda^3}2\intrn |x|^2u\dd x 
      + \frac\lambda2\intrn u|\nabla\log u|^2\dd x + \frac{\lambda^2}2\intrn x\cdot\nabla u\dd x \\
    &\qquad  + \frac12\intrn u\|\nabla^2\log u\|^2\dd x + \lambda\intrn u\Delta\log u\dd x + \frac{d\lambda^2}2\intrn u\dd x \\
    &= \frac12\intrn u\|\nabla^2\log u\|^2\dd x + \frac{\lambda^3}2\intrn |x|^2u\dd x
      -\frac\lambda2\intrn u|\nabla\log u|^2\dd x-\frac{d\lambda^2}2 \\
    &= \frac12\intrn u\|\nabla^2\log u\|^2\dd x + \lambda^3\intrn |x|^2u\dd x
      -\lambda\big[\fish_\lambda(u) - \fish_\lambda(U_\lambda)\big] - \frac{3d\lambda^2}2.
  \end{align*}
  Finally, observe that
  \begin{align*}
    \nrg_\lambda(U_\lambda) = \frac12\intrn U_\lambda\|\lambda\eins\|^2\dd x + \lambda^3\intrn|x|^2U_\lambda\dd x
    = \frac{d\lambda^2}2 + \frac{d\lambda^3}\lambda = \frac{3d\lambda^2}2.
  \end{align*}
  This yields \eqref{eq:magicF}.
\end{proof}
The combination of \eqref{eq:magicH} and \eqref{eq:magicF} suggests that for all sufficiently regular $u$,
\begin{align}
  \label{eq:ddt}
  \nrg_\lambda(u) = -\frac12\frac{\dn}{\dd r}\fish_\lambda(u) + \lambda\fish_\lambda(u) + \frac{d\lambda^2}2
  = \frac14\frac{\dn^2}{\dd r^2}\ent_\lambda(u) - \frac{\lambda}2\frac{\dn}{\dd r}\ent_\lambda(u) + \frac{3d\lambda^2}2.
\end{align}


\subsection{The displacement Hessian}
\label{sct:disphess}
The geometric idea behind the linearization by means of the displacement Hessian
is the representation of the dynamics on the space of probability measures in Lagrangian coordinates.
For the moment, let us consider general $L^2$-Wasserstein gradient flow,
written in the form of a nonlinear transport equation,
\begin{align}
  \label{eq:transport}
  \partial_tu_t = -\dv(u_t\velo[u_t])
  \quad \text{with} \quad
  \velo[u_t] = -\frac{\delta\fnc}{\delta u}\bigg|_{u=u_t}.
\end{align}
By a Lagrangian representation of a solution $u_t$ with respect to some reference measure $U$,
we mean a time-dependent diffeomorphism $X_t:\setR^d\to\setR^d$
satisfying
\begin{align}
  \label{eq:wXU}
  u_t = X_t\#U = \frac{U}{\det\dff X_t}\circ X_t^{-1}.
\end{align}
Note that there is a freedom of gauge here:
\eqref{eq:wXU} determines $X_t$ only up to a concatenation from the right
with any $t$-dependent map that leaves $U$ invariant.

Since $u_t$ satisfies the transport equation \eqref{eq:transport},
it is easily deduced that $X_t$ ``follows the vector field $\velo$''.
I.e., it satisfies the Lagrangian equation
\begin{align}
  \label{eq:Lagrange}
  \partial_tX_t  \cong \velo[u_t]\circ X_t.
\end{align}
Here $\cong$ refers to the aforementioned freedom of gauge for $X_t$:
the left and right sides in \eqref{eq:Lagrange} may differ by a vector field $\zeta_t$
that is divergence-free with respect to $U$, i.e., $\dv(U\zeta_t)=0$.

Now assume that $U$ is a stationary solution of \eqref{eq:transport};
then $X=\id$ is a stationary solution of \eqref{eq:Lagrange}.
The Wasserstein-linearization of \eqref{eq:transport} around $U$
is an appropriate linearization of \eqref{eq:Lagrange} around $\id$.
Taking into account that $\velo[U]=0$,
one obtains for any smooth, compactly supported vector field $\Xi$:
\begin{align*}
  \frac{\dn}{\dd h}\bigg|_{h=0}\big(\velo[(\id-h\zeta)\#U]\circ(\id-h\Xi)\big)
  = \nabla\left[\frac{\delta^2\fnc}{\delta u^2}\bigg|_{u=U}\dv(U\Xi)\right].
\end{align*}
Consequently, the linearized Lagrangian dynamics is given by
\begin{align}
  \label{eq:lLagrange}
  \partial_t\Xi_t \cong \nabla\left[\frac{\delta^2\fnc}{\delta u^2}\bigg|_{u=U}\dv(U\Xi)\right].
\end{align}
For definition of the displacement Hessian and the Wasserstein linearization,
one chooses a particular gauge in \eqref{eq:Lagrange} --- and consequently also in \eqref{eq:lLagrange} ---
to remove the ambiguity.
Thanks to the Brenier theorem from optimal transportation,
one may assume $X_t=\id-\nabla\varphi_t$ with some time-dependent potential $\varphi_t$.
Inserting this into \eqref{eq:lLagrange} yields
\begin{align}
  \label{eq:llLagrange}
  \partial_t\varphi_t = \frac{\delta^2\fnc}{\delta u^2}\bigg|_{u=U}\dv(U\nabla\varphi_t).
\end{align}
The operator on the right-hand side acting on $\varphi_t$ is the the negative of the displacement Hessian of $\fnc$.
More rigorously, one defines:
\begin{definition}
  Assume that $U$ is a global minimizer of $\fnc$,
  and assume further that there exists a densely defined self-adjoint linear operator $\theL$ on $H^1(\setR^d;U\leb)$
  such that, for any test function $\psi\in C^\infty_c(\setR^d)$,
  \begin{align*}
    \intrn\nabla\psi\cdot\nabla\big(\theL\psi\big) U\dd x = \frac{\dn^2}{\dd\sigma^2}\bigg|_{\sigma=0}\fnc(u_\sigma),
  \end{align*}
  where $u_\sigma$ is the solution to the transport equation $\partial_\sigma u_\sigma +\dv(u_\sigma\nabla\psi)=0$.
  Then $\theL$ is called displacement Hessian of $\fnc$ at $U$, and is denoted by $\hess_U\fnc$.  
\end{definition}


\subsection{Calculation of the displacement Hessian}
We shall now calculate the displacement Hessian for the three functionals from \eqref{eq:HFE}.
Note that they all have $U_\lambda$ as global minimizer.
\begin{proposition}
  Define the linear operator $\theL$ on $C^\infty_c(\setR^d)$ by
  \begin{align*}
    \theL\varphi := -\frac1{U_\lambda}\dv(U_\lambda\nabla\varphi).
  \end{align*}
  Then we have for all $\psi\in C^\infty_c(\setR^d)$:
  \begin{align}
    \label{eq:hessianHFE}
    \big(\hess_{U_\lambda}\ent_\lambda\big)\psi = \theL\psi,
    \quad
    \big(\hess_{U_\lambda}\fish_\lambda\psi\big) = \theL^2\psi,
    \quad
    \big(\hess_{U_\lambda}\nrg_\lambda\big)\psi = (\theL^3+\lambda \theL^2)\psi.
  \end{align}
\end{proposition}
\begin{proof}
  Let some $\psi\in C^\infty_c(\setR^d)$ be fixed.
  For the curve $u_\s$, we choose the solution of the transport problem
  \begin{align}
    \label{eq:teq}
    \partial_\s u_\s=-\dv(u_\s\nabla\psi),\quad u_0=U_\lambda.
  \end{align}
  Note that the transport vector field $\nabla\psi$
  is independent of time and smooth with compact support,
  hence the $(\s;x)\mapsto u_\s(x)$ is an everywhere positive smooth function on $\setR\times\setR^d$.
  
  For the relative entropy, recalling the representation \eqref{eq:relativeH},
  and that $\log\frac u{U_\lambda}\equiv0$ for $u=U_\lambda$,
  \begin{align*}
    \frac{\dn^2}{\dd\s^2}\bigg|_{\s=0}\ent_\lambda(u_\s)
    = \frac{\dn}{\dd\s}\bigg|_{\s=0}\intrn\left(1+\log\frac{u_\s}{U_\lambda}\right)\partial_\s u_\s\dd x
    = \intrn\frac{\big(\partial_\s\big|_{\s=0} u_\s\big)^2}{U_\lambda}\dd x.
  \end{align*}
  Substituting \eqref{eq:teq} and integrating by parts, we obtain
  \begin{align*}
    \frac{\dn^2}{\dd\s^2}\bigg|_{\s=0}\ent_\lambda(u_\s)
    = \intrn \frac1{U_\lambda}[\dv(U_\lambda\nabla\psi)]^2\dd x
    = -\intrn U_\lambda \nabla\psi\cdot\left(\frac1{U_\lambda}\dv(U_\lambda\nabla\psi)\right)\dd x.
  \end{align*}
  This gives the first identity in \eqref{eq:hessianHFE}.
  
  For the perturbed Fisher information, we start from the representation
  \begin{align*}
    \fish_\lambda(u) = \frac12\intrn u\left|\nabla\log\frac{u}{U_\lambda}\right|^2\dd x + \fish_\lambda(U_\lambda)
  \end{align*}
  that follows from \eqref{eq:magicH0} and \eqref{eq:magicH},
  and obtain
  \begin{align*}
    \frac{\dn^2}{\dd\s^2}\bigg|_{\s=0}\fish_\lambda(u_\s)
    &= \frac{\dn}{\dd\s}\bigg|_{\s=0}\intrn \left\{\frac12(\partial_\s u_\s)\left|\nabla\log\frac{u_\s}{U_\lambda}\right|^2
      +u_\s\,\nabla\left[\log\frac{u_\s}{U_\lambda}\right]\cdot\nabla\left[\frac{\partial_\s u_\s}{u_\s}\right]\right\} \dd x \\
    &= \intrn U_\lambda\left|\nabla\frac{\partial_\s\big|_{\s=0}u_\s}{U_\lambda}\right|^2\dd x.
  \end{align*}
  Hence, by \eqref{eq:teq} and two consecutive integration by parts,
  \begin{align*}
    \frac{\dn^2}{\dd\s^2}\bigg|_{\s=0}\fish_\lambda(u_\s)
    &= \intrn U_\lambda\left|\nabla\frac{\dv(U_\lambda\nabla\psi)}{U_\lambda}\right|^2\dd x \\
    &= -\intrn \frac1{U_\lambda}\dv(U_\lambda\nabla\psi)\,\dv\left[U_\lambda\nabla\left(\frac1{U_\lambda}\dv(U_\lambda\nabla\psi)\right)\right]\dd x \\
    &= \intrn U_\lambda\nabla\psi\cdot\nabla\left\{\frac1{U_\lambda}\dv\left[U_\lambda\nabla\left(\frac1{U_\lambda}\dv(U_\lambda\nabla\psi)\right)\right]\right\}\dd x,
  \end{align*}
  which confirms the second identity in \eqref{eq:hessianHFE}.
  
  Finally, for computation of the Hessian of $\nrg_\lambda$,
  we make use of \eqref{eq:magicF0} and \eqref{eq:magicF} and so obtain
  \begin{align*}
    \frac{\dn^2}{\dd\s^2}\bigg|_{\s=0}\nrg_\lambda(u_\s)
    &= \lambda \frac{\dn^2}{\dd\s^2}\bigg|_{\s=0}\fish_\lambda(u_\s) \\
    &\qquad - \frac14\frac{\dn^2}{\dd\s^2}\bigg|_{\s=0}\intrn\dv\left(u_\s\nabla\log\frac{u_\s}{U_\lambda}\right)\left|\nabla\log\frac{u_\s}{U_\lambda}\right|^2\dd x \\
    &\qquad + \frac12\frac{\dn^2}{\dd\s^2}\bigg|_{\s=0}\intrn\frac1{u_\s}\left[\dv\left(u_\s\nabla\log\frac{u_\s}{U_\lambda}\right)\right]^2\dd x 
  \end{align*}
  An quick inspection of the last two lines above reveals that there is exactly one term
  that does not vanish automatically because of $\log\frac{u_0}{U_\lambda}\equiv0$,
  namely the one where both logarithmic terms inside the last integral get differentiated.
  In combination with the calculation for $\fish_\lambda$ above, we conclude that
  \begin{equation}
    \label{eq:energywithfish}
    \begin{split}
      \frac{\dn^2}{\dd\s^2}\bigg|_{\s=0}\nrg_\lambda(u_\s)
      &= \lambda\intrn U_\lambda \nabla\psi\cdot \nabla (\theL^2\psi)\dd x
      + \intrn\frac1{U_\lambda} \left[\dv\left(U_\lambda\nabla\frac{\partial_\s|_{\s=0}u_\s}{U_\lambda}\right)\right]^2\dd x.                    
    \end{split}
  \end{equation}
  We consider the last integral, substitute \eqref{eq:teq}, and repeatedly integrate by parts:
  \begin{align*}
    &\intrn\frac1{U_\lambda} \left[\dv\left(U_\lambda\nabla\frac{\partial_\s|_{\s=0}u_\s}{U_\lambda}\right)\right]^2\dd x
      =\intrn\frac1{U_\lambda} \left[\dv\left(U_\lambda\nabla\frac{\dv(U_\lambda\nabla\psi)}{U_\lambda}\right)\right]^2\dd x \\
    &\qquad = \intrn \frac1{U_\lambda}\dv(U_\lambda\nabla\psi)\,
      \dv\left\{U_\lambda\nabla \left[\frac1{U_\lambda} \dv\left(U_\lambda\nabla\frac{\dv(U_\lambda\nabla\psi)}{U_\lambda}\right)\right]\right\}
      \dd x \\
    &\qquad = -\intrn U_\lambda\nabla\psi\cdot\nabla\left(
      \frac1{U_\lambda}\dv\left\{U_\lambda\nabla \left[\frac1{U_\lambda} \dv\left(U_\lambda\nabla\frac{\dv(U_\lambda\nabla\psi)}{U_\lambda}\right)\right]\right\}
      \right)\dd x.
  \end{align*}
  Substituting this in \eqref{eq:energywithfish} yields the final identity in \eqref{eq:hessianHFE}.
\end{proof}


\subsection{Special solutions and their linearization}
To illustrate the applicability of the linearization of \eqref{eq:QDD6intro},
we completely characterize the dynamics of \eqref{eq:QDD6intro}
and its Wasserstein linearization
\begin{align}
  \label{eq:QDD6linear}
  \partial_t\varphi_t = -\big(\theL_\lambda^3 + \lambda\theL_\lambda^2\big)\varphi_t  
\end{align}
in the invariant finite dimensional submanifold defined by affine deformations of Gaussians.
More specifically, we consider the --- non-linear and linearized --- dynamics
induced on the set of positive definite matrices $S\in\setR^{d\times d}$ and vectors $a\in\setR^d$
by means of
\begin{align}
  \label{eq:Xansatz}
  u_t = X_t\#U_\lambda
  \quad\text{with}\quad
  X_t(y) = S_t^{1/2}y+a_t.
\end{align}
We begin with the linearized dynamics.
Since we have 
\begin{align}
  \label{eq:psiansatz}
  X_t=\id-\nabla\psi_t
  \quad\text{with}\quad
  \psi_t(y) = \frac12y^T(\eins-S_t^{1/2})y - a_t^Ty,
\end{align}
we need to study solutions $\varphi_t$ to \eqref{eq:QDD6linear} of quadratic type,
\begin{align*}
  \varphi_t (x) = \frac12x^TA_tx + b_t^Tx + c_t.
\end{align*}
It is obvious that these form an invariant subspace under \eqref{eq:QDD6linear}.
For this ansatz, we obtain
\begin{align*}
  \theL_\lambda^3\varphi_t+\lambda\theL_\lambda^2\varphi_t
  = 6\lambda^3x^TA_tx + 2\lambda^3b_t^Tx - 6\lambda^2\ntr[A_t],
\end{align*}
and so it follows from \eqref{eq:QDD6linear} that
\begin{align}
  \label{eq:linAbc}
  \dot A_t = -12\lambda^3A_t, \quad \dot b_t=-2\lambda^3b_t, \quad \dot c_t = -6\lambda^2\ntr[A_t].
\end{align}
Note that the equation for $c_t$ can be neglected,
since the value of $c_t$ is irrelevant in $X_t=\id-\nabla\varphi_t$.

Now for the full nonlinear dynamics.
From the ansatz \eqref{eq:Xansatz},
we obtain that
\begin{align*}
  u_t(x) = \frac1{\sqrt{(2\pi)^d\det S_t}}\exp\left(-\frac\lambda2(x-a_t)^TS_t^{-1}(x-a_t)\right).
\end{align*}
It follows that, on the one hand,
\begin{align}
  \nonumber
  \partial_tu_t(x)
  &= \left[-\frac12\left(\frac{\dn}{\dd t}\log\det S_t\right)
    - \frac\lambda2 (x-a_t)^T\frac{\dn}{\dd t}\big(S_t^{-1}\big)(x-a_t)
    + \lambda(x-a)^TS_t^{-1}\dot a_t\right]
    u_t(x) \\
  \label{eq:special1}
  &= \left[
    \frac\lambda2(x-a)^TS_t^{-1}\dot S_tS_t^{-1}(x-a_t) + \lambda(x-a_t)^TS_t^{-1}\dot a_t
    -\frac12\ntr[S_t^{-1}\dot S_t]
    \right]u_t(x).
\end{align}
And on the other hand,
using that $\nabla^2u_t(x) = (\lambda^2S_t^{-1}(x-a_t)\otimes S_t^{-1}(x-a_t)-\lambda S_t^{-1})u_t(x)$,
and also that $\nabla^2\log u_t \equiv -\lambda S_t^{-1}$,
we obtain
\begin{align*}
  \Phi(u_t) = \frac32\lambda^2\|S_t^{-1}\|^2 -\lambda^3(x-a_t)^TS_t^{-3}(x-a_t),
\end{align*}
and thus further
\begin{align}
  \nonumber
  &\dv\big(u_t\big[\nabla\Phi(u_t)+2\lambda^3x\big]\big)
    =\dv\big(u_t\big[2\lambda^3(\eins-S_t^{-3})(x-a_t)+2\lambda^3a_t\big]\big) 
  \\
  \label{eq:special2}
  &= \left(
  2 \lambda^3\ntr[\eins-S_t^{-3}]
  -2\lambda^4(x-a_t)^TS_t^{-1}a_t
  -2\lambda^4(x-a_t)^T(S_t^{-1}-S_t^{-4})(x-a_t)
  \right)u_t(x).
\end{align}
By equating the right-hand sides of \eqref{eq:special1} and \eqref{eq:special2},
we see that \eqref{eq:QDD6intro} is satisfied if and only if
\begin{align*}
  \dot S_t = -4\lambda^3(S_t-S_t^{-2}),
  \quad
  \dot a_t = -2\lambda^3a_t,
  \quad
  \ntr[S_t^{-1}\dot S_t] = -4\lambda^3\ntr[\eins-S_t^{-3}].
\end{align*}
Note that the last of these three differential equations is a trivial consequence of the first.
Note further that the equations for $a_t$ and $S_t$ are easy to solve:
\begin{align}
  \label{eq:aS}
  a_t = e^{-2\lambda^3t}a_0, \quad
  S_t = \big(\eins+e^{-12\lambda^3t}(S_0^3-\eins)\big)^{1/3};
\end{align}
the third root is well-defined since the expression
in the brackets is a positive definite matrix at each $t\ge0$.

In view of \eqref{eq:psiansatz}, the quantities $(A_t,b_t)$ and $(S_t,a_t)$ are related by
\begin{align*}
  S_t = (\eins - A_t)^2, \quad a_t = -b_t.
\end{align*}
It is thus obvious that the linear ODEs \eqref{eq:linAbc}
capture the correct asymptotic behaviour of $S_t$ and $a_t$ as $t\to\infty$.
It should be noted that the $d$-fold eigenvalues $2\lambda^3$ and $12\lambda^3$
consistute the lowest eigenvalues of $\hess_{U_\lambda}\nrg_\lambda$'s spectrum;
the next eigenvalue is $36\lambda^3$.
The interpretation in terms of higher-order asymptotics would be this:
an arbitrary initial datum $u_0$ that sufficiently close to $U_\lambda$ in an appropriate sense
can be modified by an suitable transformation of the form
\begin{align*}
  \tilde u_0 = G\#u_0, \quad\text{with}\quad
  G(x) = S^{1/2}x + a,
\end{align*}
such that the corresponding solution $\tilde u_t$ converges to $U_\lambda$
at an exponential rate of $36\lambda^3$.
The rigorous proof of such a result is currently out of reach.
For a related result on the porous medium equation, we refer to \cite{Koch}.

                
\subsection{Asymptotic self-similarity}
\label{sct:selfsim}
We discuss the following consequence of Conjecture \ref{cjt:two}.
\begin{corollary}
  If conjecture \ref{cjt:two} is true,
  then for any initial datum $u_0$ of finite second moment and finite entropy,
  there is a corresponding weak solution $u$ to the initial value problem for \eqref{eq:QDD6intro} with $\lambda = 0$
  that approaches the self-similar solution $u_*$ from \eqref{eq:selfsim} at algebraic rate $t^{-1/6}$.
  That is,
  \begin{align}
    \label{eq:ssasymptotics}
    \big\|u(t;\cdot)-u_*(t;\cdot)\big\|_{L^1(\setR^d)}\le C(u_0)(1+12t)^{-1/6}
  \end{align}
\end{corollary}
\begin{proof}
  By the usual rescaling for homogeneous parabolic equations,
  we relate the solution $u$ for $\lambda=0$ to a solution $v$ for $\lambda=1$. 
  Specifically, we set
  \begin{align*}
    \kappa(t) := (1+12t)^{1/6},
    \quad
    \tau(t) = \log \kappa(t),
  \end{align*}
  and then introduce $v=v(s;y)$ implicitly via
  \begin{align}
    \label{eq:v2u}
    u(t;x) = \kappa(t)^{-d}v\big(\tau(t);\kappa(t)^{-1}x\big).
  \end{align}
  Note that $u(0;x)=v(0;x)$.
  Now let $v$ be the solution to \eqref{eq:QDD6intro} with $\lambda=1$ according to Theorem \ref{thm:existence}
  for the initial condition $v(0;x)=u_0(x)$.
  Performing a change of variables under the integral in \eqref{eq:N}
  it is easily seen that $u$ satisfies the weak formulation \eqref{eq:weak} with $\lambda=0$.

  To conclude the self-similar asymptotics \eqref{eq:ssasymptotics} via Conjecture \ref{cjt:two},
  we write \eqref{eq:5} for $v$ in place of $u$, and with $\lambda=1$,
  and perform a change of variables under the integral:
  \begin{align*}
    C(u^0)(1+12t)^{-1/6} = C(u^0)e^{-\tau(t)}
    &\ge \big\|v(\tau(t);\cdot)-U_1\big\|_{L^1(\setR^d)} \\
    & = \intrn |v(\tau(t);y)-U_1(y)|\dd y \\
    & = \intrn |v(\tau(t);\kappa(t)^{-1}x)-U_1(\kappa(t)^{-1}x)|\,\kappa(t)^{-d}\dd x \\
    &= \intrn |u(t;x)-u_*(t;x)|\dd x
    = \|u(t;\cdot)-u_*(t;\cdot)\|_{L^1(\setR^d)},
  \end{align*}
  where we have used the definition of $u_*$ in \eqref{eq:selfsim}.
\end{proof}



\begin{appendix}

  \section{Two elementary inequalities for sums of powers}
  \begin{lemma}
    \label{lem:elementary1}
    For each $n=1,2,\ldots$,
    \begin{align}
      \label{eq:elementary1}
      (n\tau)^{5/6}\le \tau\sum_{N=1}^n(N\tau)^{-1/6}
    \end{align}
  \end{lemma}
  \begin{proof}
    We prove \eqref{eq:elementary1} by induction on $n$.
    For $n=1$, one has equality.
    Now assume that \eqref{eq:elementary1} holds for $n\ge1$;
    we need to show that
    \begin{align*}
      ((n+1)\tau)^{5/6}-(n\tau)^{5/6}\le ((n+1)\tau)^{-1/6}\tau.
    \end{align*}
    By the ``below tangent formula'' for the concave function $s\mapsto s^{5/6}$,
    we have
    \begin{align*}
      ((n+1)\tau)^{5/6} \le (n\tau)^{5/6} + \frac56(n\tau)^{-1/6}\tau.
    \end{align*}
    We conclude by observing that
    \begin{align*}
      \frac56\left(\frac{n+1}n\right)^{1/6} \le \frac56 2^{1/6} < 1
    \end{align*}
    for all $n=1,2,\ldots$.
  \end{proof}
  \begin{lemma}
    \label{lem:elementary2}
    For integers $\overline M>\underline M\ge 1$,
    \begin{align}
      \label{eq:elementary2}
      \frac\tau{12}\sum_{n=\underline M}^{\overline M-1}(n\tau)^{-5/6} \le (\overline M\tau)^{1/6}-(\underline M\tau)^{1/6}.
    \end{align}
  \end{lemma}
  \begin{proof}
    Inequality \eqref{eq:elementary2} clearly follows from
    \begin{align*}
     \frac\tau{12}(n\tau)^{-5/6} \le ((n+1)\tau)^{1/6}-(n\tau)^{1/6}
    \end{align*}
    for all $n=1,2,\ldots$.
    This is a consequence of the ``below tangent formula'' for the concave function $s\mapsto s^{1/6}$,
    \begin{align*}
      (n\tau)^{1/6} \le ((n+1)\tau)^{1/6} - \frac16((n+1)\tau)^{-5/6}\tau,
    \end{align*}
    in combination with the observation that
    \begin{align*}
      \frac16\left(\frac n{n+1}\right)^{5/6} \ge \frac16 2^{-5/6} > \frac1{12}.
    \end{align*}
  \end{proof}

  \section{A convergence theorem for powers}
  \begin{theorem}
    \label{thm:interpol}
    For $0 < \beta < \gamma < \alpha < \infty$ with $\alpha p = \beta q = \gamma r$ 
    and a sequence of nonnegative functions $(u_n)_{n \ \mathbb{N}}$ on $(0,T) \times \rn$
    such that
    \begin{enumerate}
    \item [(i)] $u_n^\alpha \rightarrow u^\alpha $ strongly in $L^p(0,T;W^{1,p}(\rn))$ 
    \item[(ii)] $(u_n^\beta)_{n \in \mathbb{N}}$ is bounded in $L^q(0,T;W^{1,q}(\rn))$
    \end{enumerate}
    we have (up to subsequences) $u_n^\gamma \rightarrow u^\gamma$ strongly in $L^r(0,T;W^{1,r}(\rn))$ .
  \end{theorem}
  We rephrase a slight variant, respectively extension, of the proof of \cite[Proposition 6.1]{JuMi}. 
  \begin{proof}
    At first assume that the sequence is uniformly bounded away from 0, i.e.
    $u_n \geq \eps \; \forall n \in \mathbb{N}$ for some $\epsilon > 0$, 
    and that all $u_n$ have support on a ball with radius $R > 0$, $\Omega := B_0(R)$.
    By the "missing term in Fatou's Lemma" 
    it is enough to show pointwise convergence and convergence of the norms 
    for the strong convergence of $(u_n^\gamma )_{n \in \mathbb{N}}$ in  $L^r(0,T;W^{1,r}(\Omega))$.
    
    For this, define $\mu_n, \mu$ to be the measures with densities $u_n , u$ (with respect to the Lebesgue measure). 
    As $(u_n)_{n \in \mathbb{N}}$ converges in $L^1(0,T;L^1(\Omega))$, 
    which of course implies weak convergence in $L^1(0,T;L^1(\Omega))$, 
    $(\mu_n)_{n \in \mathbb{N}}$ converges narrowly to $\mu$ on $(0,T)\times \Omega$ 
    (since $C_b((0,T) \times \Omega) \subset L^\infty((0,T)\times \Omega)$). 
    Moreover for $v_n := \nabla u_n / u_n$ we have 
    \begin{equation}\label{eq: existence_aprioriEstimates_ju_sobolevInt}
      \int_{(0,T)\times\Omega} |v_n|^p \, d\mu_n(x) 
      = \int_{(0,T)\times\Omega} |\nabla u_n|^p u_n^{p(\alpha-1)}  \, dx 
      = \alpha^{-p} \int_{(0,T)\times\Omega} |\nabla u_n^\alpha|^p \, dx
    \end{equation}
    and hence $v_n \in L^p(\mu_n, (0,T)\times\Omega)$. 
    
    Further, (\ref{eq: existence_aprioriEstimates_ju_sobolevInt}) and assumption $(i)$ imply
    \begin{equation*}
      \limsup_{n \rightarrow 0} \norm{v_n}_{L^p(\mu_n, (0,T)\times\Omega)} = \norm{v}_{L^p(\mu, (0,T)\times\Omega)},
    \end{equation*}
    hence $(v_n)_{n \in \mathbb{N}}$ converges in the sense of \cite[Definition 5.4.3.]{AGS} 
    strongly in $L^p(\mu, (0,T)\times\Omega)$, 
    which in addition to the narrow convergence of $(\mu_n)_{n \in \mathbb{N}}$,
    implies the narrow convergence of the transport plans $\gamma_n = (\id \times v_n)\#\mu_n$
    to $\gamma = (\id \times v)\#\mu$ in $\propp(((0,T)\times\Omega) \times ((0,T)\times\Omega))$,
    as stated in \cite[Theorem 5.4.4.]{AGS}.
    Since the $p$th and $q$th moment of $\mu_n$ are uniformly bounded,
    by applying \cite[Theorem 5.1.7.]{AGS} to $(\gamma_n)_{n \in \mathbb{N}}$ we find
    \begin{equation*}
      \lim\limits_{n \rightarrow \infty} \int_{(0,T)\times\Omega} f(x, v_n(x)) \, d\mu_n(x)
      =  \int_{(0,T)\times\Omega} f(x, v(x)) \, d\mu(x)
    \end{equation*}
    for functions $f$ with at most $r$-growth.
    Hence choosing $f(x,y) = |y|^r$ yields the norm convergence:
    \begin{equation*}
      \lim\limits_{n \rightarrow \infty} \int_{(0,T)\times\Omega} |\nabla u_n^\gamma|^r \,
      dx  = \lim\limits_{n \rightarrow \infty} \int_{(0,T)\times\Omega} |v_n(x)|^r \, d\mu_n(x)
      = \int_{(0,T)\times\Omega} |\nabla u^\gamma|^r \, dx.
    \end{equation*}
    By assumption $(i)$ we can extract pointwise convergent subsequences (not relabelled)
    \begin{align*}
      u^\alpha_n(x) &\rightarrow u^\alpha(x) \quad  a.e. \, x \in (0,T) \times \Omega\\
      \nabla u^\alpha_n(x) &\rightarrow \nabla u^\alpha(x) \quad  a.e. \, x \in (0,T) \times \Omega
    \end{align*}
    which, taking the strict positivity of the functions into account, yields
    \begin{align*}
      u^\gamma_n(x) &\rightarrow u^\gamma(x) \quad a.e. \, x \in (0,T) \times \Omega\\
      \nabla u^\gamma_n(x) =  u_n^{\gamma - \alpha}(x) \nabla u^\alpha_n(x)
                    &\rightarrow  u^{\gamma - \alpha}(x) \nabla u^\alpha(x)
                      = \nabla u^\gamma(x) \quad  a.e. \, x \in (0,T) \times \Omega
    \end{align*}
    and hence finally the strong convergence of the subsequence.
    
    For a nonnegative sequence $(u_n)_{n \in \mathbb{N}}$ we apply the above result to the truncated functions
    \begin{equation*}
      u_{n, \epsilon}(x) := \max\{ u_n(x), \epsilon \} \quad a.e. \, x \in (0,T) \times \Omega
    \end{equation*}
    for a sequence $\epsilon \downarrow 0$.
    Using the norm convergence of $(u_n^\gamma)_{n \in \mathbb{N}}$ and the truncation property of $W^{1,r}(\Omega)$,
    we have 
    \begin{equation}
      \norm{u_{n}^\gamma - u_{n, \epsilon}^\gamma }_{L^r(0,T;W^{1,r}(\Omega))} = 
      \norm{u_{n}^\gamma - (u^\gamma)_{n, {\epsilon}^\gamma}}_{L^r(0,T;W^{1,r}(\Omega))} \rightarrow 0
    \end{equation} 
    for $\eps \rightarrow 0$ and thus the result also holds for nonnegative functions.
    
    To extend the argument to the whole $\rn$, one argues by
    approximating $u_n$ by the sequence $u_{n, l} = \chi_l
     u_n$ for $\chi_l(x) = \chi(|x|/l)$
    where $\chi\in C^\infty_c(\setR^d)$ is a cut-off function
    with $0\le\chi\le1$, where $\chi(x)\equiv1$ for $|x|\le1$ and $\chi(x)=0$ for $|x|\ge2$.
    Passing to subsequences and using the former calculations,
    by diagonalization one finds a sequence converging on every ball $B_0(l)$ for $l \in \mathbb{N}$.
    With the norm boundedness of $u^\gamma_n$ in $L^r(0,T;W^{1,r}(\Omega))$
    obtained by using the representation given in (\ref{eq: existence_aprioriEstimates_ju_sobolevInt})
    and interpolating between the uniformly bounded $\norm{v_n}_{L^p(\mu, (0,T)\times\rn)}$ and $\norm{v_n}_{L^q(\mu, (0,T)\times\rn)}$
    one eventually concludes the argument for functions on the whole $\rn$.
  \end{proof}

  \section{Integration by parts}
  For convenience of the reader, we recall the basic rule for integration by parts.
  \begin{lemma}
    \label{lem:Gauss}
    Let $f,g\in L^1(\setR^d)$ be given,
    and assume that there exists a vector field $\velo\in L^1(\setR^d;\setR^d)$
    such that $f=g+\dv\velo$ in the sense of distributions.
    Then
    \begin{align}
      \label{eq:f=g}
      \intrn f\dd x = \intrn g\dd x.
    \end{align}
  \end{lemma}
  \begin{proof}
    Let $\chi\in C^\infty_c(\setR^d)$ be a cut-off function
    with $0\le\chi\le1$, with $\chi(x)\equiv1$ for $|x|\le1$, and with $\chi(x)=0$ for $|x|\ge2$.
    Defining as usual $\chi_R(x):=\chi(x/R)$ for $R>1$,
    we obtain
    \begin{align*}
      \intrn \chi_Rf\dd x = \intrn \chi_Rg\dd x + \langle \dv\velo,\chi_R\rangle
      = \intrn  g\chi_R\dd x-\intrn\velo\cdot\nabla\chi_R\dd x.
    \end{align*}
    On the one hand,
    since $g\in L^1(\setR^d)$, we have that $g\chi_R\to g$ in $L^1(\setR^d)$ as $R\to\infty$ by dominated convergence.
    On the other hand,
    since $\velo\in L^1(\setR^d;\setR^d)$ and $\nabla\chi_R(x)=\frac1R\nabla\chi(x/R)$,
    we have that $\velo\cdot\nabla\chi_R\to0$ in $L^1(\setR^d)$ as $R\to\infty$, again by dominated convergence.
    This shows \eqref{eq:f=g}.
  \end{proof}
  An application of this variant of integration by parts is the following identity, taken from \cite[Theorem 3.1]{GST}:
  \begin{align*}
    4\intrn |\nabla\sqrt[4]u|^4\dd x
    = 2\intrn \nabla\sqrt[4]u\cdot\nabla^2\sqrt u\cdot\nabla\sqrt[4]u \dd x
    + \intrn \Delta\sqrt u |\nabla\sqrt[4]u|^2\dd x
  \end{align*}
  for all positive $u\in L^1(\setR^d)$ with $\sqrt u\in W^{2,2}(\setR^d)$.
  The respective vector field is given by $\velo = |\nabla\sqrt[4]u|^2\nabla\sqrt u$,
  which is $L^1(\setR^d;\setR^d)$ since $\nabla\sqrt u\in L^2(\setR^d;\setR^d)$ and $\nabla\sqrt[4]u\in L^4(\setR^d)$.
  The relation $f+\dv\velo=g$ is easily seen for positive and smooth functions $u$,
  and carries over to the afore mentioned more general $u$ via local approximation.
  
  \section{Arzel\'{a}-Ascoli-Theorem}
  For enhanced self-containment, we replicate the statement of the generalized Arzel\'{a}-Ascoli Theorem from \cite[Proposition 3.3.1]{AGS},
  which has been essential in the proof of Lemma \ref{lem:limit}.
  \begin{lemma}\label{lem:arzela-ascoli}
    Let $(\mathcal{I}, d)$ be a complete metric space and $\sigma$ an Hausdorff topology on $\mathcal{I}$ compatible with $d$ in the sense that for sequences $(x_n), (y_n) \subset \mathcal{I}$ 
    \begin{align*}
      x_n \xrightarrow{d} x \quad &\Longrightarrow \quad  x_n \xrightarrow{\sigma} x\\
      \textnormal{and} \quad (x_n,y_n) \xrightarrow{\sigma} (x,y) \quad &\Longrightarrow \quad \liminf_{n \rightarrow \infty} d(x_n,y_n) \geq d(x,y).
    \end{align*}
    Further, let $K \subset \mathcal{I}$ be a sequentially compact set w.r.t. $\sigma$ and let $u_n: [0,T] \rightarrow \mathcal{I}$ be curves such that
    \begin{gather*}
      u_n(t) \in K \quad \forall n \in \mathbb{N}, \quad t \in [0,T],\\
      \limsup_{n \rightarrow \infty} d(u_n(s), u_n(t)) \leq w(s,t) \quad \forall s,t \in [0,T],
    \end{gather*}
    for a (symmetric) function $w: [0,T]\times [0,T] \rightarrow [0, +\infty)$, such that
    \begin{equation*}
      \lim_{(s,t) \rightarrow (r,r)} w(s,t) = 0 \quad \forall r \in [0,T]\setminus \mathcal{I},
    \end{equation*}
    where $\mathcal{I}$ is an (at most) countable subset of $[0,T]$. Then there exists an increasing subsequence $k \mapsto n(k)$ and a limit curve $u: [0,T] \rightarrow \mathcal{I}$ such that
    \begin{equation*}
      u_{n(k)} \xrightarrow{\sigma} u(t) \quad \forall t \in [0,T], \quad \textnormal{$u$ is $d$-continuous in } [0,T]\setminus \mathcal{I}.
    \end{equation*}
  \end{lemma}
  In the proof of Lemma \ref{lem:limit}, this result is applied as follows.
  The complete metric space is $(\propo,\wass)$,
  the auxiliary topology is the one induced by narrow convergence;
  we have recalled the compatibility conditions in Section \ref{sct:background}.
  The compact subset $K\subset\propo$ is formed by the probability measures
  satisfying the uniform bound on the second moment from \eqref{eq:impromom};
  sequential compactness of $K$ is a consequence of Prokhorov's theorem.
  The modulus $w$ of continuity is given by the right-hand side in the H\"older estimate \eqref{eq:holderall} at $\tau=0$.
  Notice that the conclusion above guarantees convergence in $\sigma$, i.e., narrowly,
  not necessarily in $\wass$.
\end{appendix}

\bibliographystyle{plain}
\bibliography{QDD6}

\end{document}